\documentclass[9pt, a4paper, oneside]{amsart}
\usepackage{amsfonts,amsthm}
 \usepackage{amsmath}
 \usepackage{amssymb}
 \usepackage[latin1]{inputenc}
 \usepackage[all]{xypic}
 \usepackage{mathrsfs}
 \usepackage{euscript}
 \usepackage[all]{xy}
 \usepackage{enumerate}
 \usepackage{graphicx}
\usepackage{geometry}
 \font\Bbb=msbm10

 \def\R{\hbox{\Bbb R}}

 \newcommand{\Hom}{\operatorname{Hom}}

 \newcommand{\CoDer}{\operatorname{CoDer}}

\newcommand{\Tpoly}{\mathcal T_{poly}}
\newcommand{\Dpoly}{\mathcal D_{poly}}

 \newcommand{\Id}{\operatorname{Id}}
\newcommand{\ide}{\operatorname{id}}
\newcommand{\ot}{\otimes}
\newcommand{\bt}{\bigotimes}

 \newtheorem{theorem}{Theorem}[section]
 \newtheorem{lemma}[theorem]{Lemma}
 \newtheorem{corollary}[theorem]{Corollary}
 \newtheorem{proposition}[theorem]{Proposition}

 \newtheorem{remark}[theorem]{Remark}

 \newtheorem{example}[theorem]{Example}
 
 \newtheorem{definition}[theorem]{Definition}

\makeatletter
\newcommand\ackname{Acknowledgements}
\if@titlepage
  \newenvironment{acknowledgements}{%
      \titlepage
      \null\vfil
      \@beginparpenalty\@lowpenalty
      \begin{center}%
        \bfseries \ackname
        \@endparpenalty\@M
      \end{center}}%
     {\par\vfil\null\endtitlepage}
\else
  
\fi
\makeatother

\makeatletter
\newcommand\abname{Abstract}
\if@titlepage
  \newenvironment{abstractenv}{%
      \titlepage
      \null\vfil
      \@beginparpenalty\@lowpenalty
      \begin{center}%
        \bfseries \ackname
        \@endparpenalty\@M
      \end{center}}%
     {\par\vfil\null\endtitlepage}
\else
  \newenvironment{abstractenv}{%
      \if@twocolumn
        \section*{\abstractname}%
      \else
        \small
        \begin{center}%
          {\bfseries \abname\vspace{-.5em}\vspace{\z@}}%
        \end{center}%
        \quotation
      \fi}
      {\if@twocolumn\else\endquotation\fi}
\fi
\makeatother

\begin{document}

\title[Configuration spaces and $A_\infty$-homotopies]{Compactified configuration space of points on a line and homotopies of $A_\infty$ morphisms}
 
\author{Theo Backman}
\address{Department of Mathematics \\ Stockholm University \\ 106 91 \\ Stockholm \\ Sweden}
\email{theob@math.su.se}

\maketitle

    \setcounter{section}{-1}
\begin{abstractenv}
We construct a configuration space model for a particular 2-colored differential graded operad encoding the structure of two $A_\infty$ algebras with two $A_\infty$ morphisms and a homotopy between the morphisms. The cohomology of this operad is shown to be the well-known 2-colored operad encoding the structure of two associative algebras and of an associative algebra morphism between them.
\end{abstractenv}

\bigskip

\begin{section}{Introduction}
Many important algebraic operads can be reinterpreted as operads of chains on a topological operad. One of the first and most important operads is the topological operad of little disks. The associated {\em chain} operad of little disks was studies in a paper by Cohen \cite{Co}, where he proved that its homology operad coincides with the operad of {\em Gerstenhaber algebras}. Another example is given by the operad of little intervals $D_1(\R)$. The $n$-th part of this operad is (roughly speaking) given by the space of embeddings of $n$ copies of $\R$ into $\R$ such that the image intervals are disjoint. The representations of $D_1(\R)$ are the same thing as $A_\infty$ spaces and the chains of this topological operad is a differential graded operad that is quasi-isomorphic to the operad $\mathcal A_\infty.$ There is however another very useful way to connect the theory of $A_\infty$ algebras to the theory of geometric operads. 
 
Consider $n$ points on the real line modulo the action of the affine group; $x \mapsto \lambda x + c$ where $\lambda \in \R^{+}$ and $c \in \R.$ The space of such configurations of points is an $n-2$ dimensional manifold $C_n(\R)$. This manifold $C_n(\R)$ can be suitably compactified into a closed manifold with corners $\overline{C}_n(\R)$ in a such a way that the whole family $\{ \overline{C}_n(\R)\}_{n \geq 2}$ gives us an operad in the category of smooth manifolds with corners. The associate operad of fundamental chains is identical to the operad of $\mathcal A_\infty$ algebras. Note that in this approach we get a geometric interpretation of the $\mathcal A_\infty$ operad in terms of manifolds - {\em not just topological spaces!}  Therefore this approach gives us new mathematical tools when studying strongly homotopy algebras, as for example, manifolds with corners are always equipped with sheafs of differential forms which one can integrate and which obey the Stokes theorem. Therefore such an interpretation of an algebraic operad in terms of an operad of configuration spaces opens up the possibility of obtaining transcendental results; those results that cannot be achieved just through homological algebra and perturbative methods.  There are two such famous transcendental results due to Kontsevich. 

In the 90s Kontsevich made a ground breaking contribution to the field of mathematical physics by proving his {\em Formality conjecture} \cite{Ko1}. The result gives an $L_\infty$ quasi-isomorphism $$ \mathcal K: (\Tpoly(\R^d),[-,-]_{SN},d=0) \longrightarrow (\Dpoly(\mathcal O (\R^d)),[-,-]_G,d_H)$$
from the Lie algebra of polyvectorfields on $\R^d$ equipped with the Schouten-Nijenhuis bracket and the trivial differential to the Lie algebra of polydifferential operators on smooth functions on $\R^d$ equipped with the Gerstenhaber bracket and the Hochschild differential. The proof is via an explicit construction of the map given by integrals on configuration spaces. Kontsevich formality theorem can be formulated as a morphism of operads, i.e. as a morphism from the operad of fundamental chains of Kontsevich configuration spaces to the operad of Kontsevich graphs. 

The second such transcendental result due to Kontsevich gives an explicit proof of the formality of the little disk operad \cite{Ko2}.

The 2-colored operad $\mathcal Mor(As)_\infty$ also admits a nice configuration space model \cite{Me}. In this case one studies compactifications of configurations of $n$ points on the line modulo just translation. The major achievement of this paper is a construction of a configurations space model for the 2-colored dg operad $\mathcal Ho(As)_\infty$ which controls a pair of $A_\infty$ algebras, $(V,\mu^V),$ $(W,\mu^W)$, a pair of $A_\infty$ morphism between them $f,g:(V,\mu^V) \to (W, \mu^W)$ and a homotopy between these morphism, $$h: f \sim g.$$ Put another way a representation of our 2-colored operad is a diagram in the category of $A_\infty$ algebras like this:

\[\xymatrix{
\left(V,\mu^V\right)
\ar@/^1.5pc/[rr]_{\quad}^{f}="1"
\ar@/_1.5pc/[rr]_{g}="2"
&& \left(W,\mu^W\right)
\ar@{}"1";"2"|(.2){\,}="7"
\ar@{}"1";"2"|(.8){\,}="8"
\ar@{=>}"7" ;"8"^{h} } \]

This time we study a suitable compactification of $n$ different points on a line (without taking a quotient with respect the action of any Lie group). In this construction we recover the earlier configuration space models for $A_\infty$ algebras and $A_\infty$ morphisms which we discussed above. We also calculate the cohomology of the 2-colored dg operad $\mathcal Ho(As)_\infty$ and show that it is equal to $\mathcal Mor(As)$. This result proves that $\mathcal Ho(As)_\infty$ is a non-minimal model of $\mathcal Mor(As).$ This result also implies, after some additional work, that $\mathcal Ho(As)_\infty$ is a non-minimal quasi-free model of the 2-colored dg operad $\mathcal Ho(As)$ which is, by definition, the operad encoding the structure of two dg associative algebras, two algebra morphisms between them and a homotopy between these two morphisms.

\begin{subsection}{Outline}
The paper is divided into three sections.

In section 1 an exposition of some basic facts of algebra, and specifically $A_\infty$ algebras, are given. We will describe a sort of translation between two equivalent definitions of $A_\infty$ algebras and related notions. 

In section 3 we will introduce two families of configuration spaces, describe how they can be compactified to form operads of topological spaces and describe how their associated face complexes are related to $A_\infty$ algebras.

In section 4 some novel results are given. We describe a compactification of the family of configuration space of $n$ points on a line. This space is described to be an operad of smooth manifolds with corners. The face complex is identical to the two-colored operad of homotopies between a pair of $A_\infty$ morphisms, $\mathcal Ho(As)_\infty$. We also determine the cohomology of $\mathcal Ho(As)_\infty$ and prove that it is $\mathcal Mor(As)$

\end{subsection}

\end{section}

 \bigskip 

\begin{section}{Coalgebras and $A_\infty$ algebras}

Through out this text we let $k$ be a field of characteristic zero.

\begin{definition}
A coassociative graded coalgebra $C$ over the field $k$ is a graded $k$-vector space with a degree zero comultiplication map $\Delta : C \to C\ot C$ such that the following diagram commutes:
\[
\xymatrix{ C \ot C \ot C \ar@{<-}[d]^-{\Delta \ot \ide} \ar@{<-}[r]^-{\ide \ot \Delta} & C \ot C \ar@{<-}[d]^-{\Delta} \\
C\ot C \ar@{<-}[r]^-{\Delta} & C
}.\]
\end{definition}
 We say that a coassociative coalgebra is counital if there exist a map $\eta:k \gets C,$ called the counit, such that the following diagram commutes:
\[
\xymatrix{k \ot C \ar@{<-}[rr]^-{\eta \ot \ide } \ar@{<-}[dr]^-{\cong} & & C \ot C \ar@{->}[dl]^-{\Delta} \\
& C \ar@{->}[dl]^-{\cong} \ar@{<-}[dr]^-{\Delta} & \\
 k \ot C \ar@{<-}[rr]^-{\ide \ot \eta} & & C \ot C
}
\]
\begin{definition}
Let $(C,\Delta_C,\eta_C)$ and $(D,\Delta_D,\eta_D)$ be coalgebras. A map of coalgebras $F: C \to D$ is a $k$-linear map such that the following diagrams commutes
\[
\xymatrix{ C  \ar@{->}[d]^-{\Delta_C} \ar@{->}[r]^-{F} & D \ar@{->}[d]^-{\Delta_D} \\
C\ot C \ar@{->}[r]^-{F \ot F} & D \ot D
} \qquad
\xymatrix{ k \ar@{<-}[r]^-{\eta_C} \ar@{->}[d]^-{\ide} & C \ar@{->}[d]^-{F} \\
k \ar@{<-}[r]^-{\eta_D} & D
}\]
 \end{definition}

\begin{definition}
Let $C$ be a graded coalgebra. A linear map $b \in \Hom^n(C,C)$ such that the following diagram commutes
\[
\xymatrix{
C \ar@{->}[rr]^-{b} & & C  \\
C \ot C \ar@{<-}[u]^-{\Delta} \ar@{<-}[rr]^-{\ide\ot b + b \ot \ide} & & C \ot C \ar@{<-}[u]^-{\Delta_C}
}
\]
Is called a degree $n$ coderivation. The set of degree $n$ coderivations of a coalgebra $C$ form a $k$-vector space and is denoted $\CoDer^n(C).$ 
\end{definition}

\begin{remark}
A degree one coderivation $b$ such that $b^2=b \circ b =0$ is called a codifferential.
\end{remark}

\begin{definition}
The pair $(C,b)$ where $C$ is a graded coalgebra and $b$ codifferential of $C$ is called a differential graded coalgebra, or just dg coalgebra.

Morphism of dg coalgebras are morphism of graded coalgebras which commute with codifferentials.
\end{definition}

\begin{definition}\textbf{Tensor coalgebras.}
Let $V$ be a graded vector space. The tensor coalgebra $T_cV$ is as a vector space the direct sum $\oplus_{k \geq 0} V^{\ot i},$ where $V^{\ot i}$ is the $i$-times iterated tensor product with itself, $$V^{\ot i} = \underbrace{V \ot \ldots \ot V}_{i \mbox{\scriptsize{-times}}}.$$  $T_c V$ can be given a coalgebra structure with the coproduct map $$\Delta: T_c V  \to T_c V \ot T_c V$$ given on summand $T_c^n V = V^{\ot n}$ as $$\Delta: (v_1, \ldots, v_n) \to \sum_{i=0}^n (v_1,\ldots,v_i) \ot (v_{i+},\ldots,v_n), $$
where the term for $i=0,n$ are $1 \ot (v_1,\ldots,v_n)$ and $(v_1,\ldots,v_n) \ot 1$ inside $V^{\ot 0} \ot V^{\ot n}$ and $V^{\ot n} \ot V^{\ot 0},$ respectively. 

The reduced tensor coalgebra $\overline{T}_c V$ is as a vector space the direct sum $\oplus_{i \geq 1} V^{\ot i},$ with coproduct $${\Delta}:   \overline{T}_c V  \to \overline{T}_c V \ot \overline{T}_c V$$ given, as above, on summands as $${\Delta}: (v_1,\ldots,v_n) \to \sum_{i=1}^{n-1} (v_1 \ldots v_i) \ot (v_{i+1}, \ldots, v_n).$$
\end{definition}
\begin{remark}
 From the coproduct we define the partial coproducts: 
$$ \Delta^{a,b}_{a+b} := V^{\ot (a+b)} \hookrightarrow  \overline{T}_c V \overset{\Delta}{\longrightarrow } \overline{T}_c V \bt \overline{T}_c V \to V^{\ot a} \bt V^{\ot b}.$$ This will be a convenient short hand in many of the proofs of this section.
\end{remark}

\begin{proposition}
\label{coder}
A map of vector spaces $b:\overline{T}_c V \to V$ can be lifted to a unique coderivation of coalgebras, $$B:\overline{T}_c V \to \overline{T}_c V,$$
such that $pr_1 \circ B = b,$ when $pr_1$ is the natural projection $\overline{T}_c V \to V$. If $B^m_n$ denotes the composition $$ B^m_n :V^{\ot n} \hookrightarrow  \overline{T}_c V \overset{B}{\to } \overline{T}_c V \to V^{\ot m},$$
then the explicit formula for $B^m_n$ is given by $$B^m_n = \begin{cases} 0 \mbox{ if } n<m \\ \sum_{i+j=m-1} \Id^{\ot i} \ot b_{n+1-m} \ot \Id^{\ot j} \end{cases}$$
where $b_a := b|_{V^{\ot a}}.$ Furthermore, the map $B$ is recovered as a product; $$B= \prod_{n \geq 1} B^n \qquad  B^n= \prod_{m \geq 1} B^n_m,$$ and note that $B_a^1 = b_a$ 
\end{proposition}
\begin{proof}
The proof is by induction. The case $B^1_n$ is clear from the projection property. Assume that for all $m < M$ we have that $B^m_n$ is given by the formula. The equation $$\Delta B = (\Id \bt B+B\bt \Id ) \Delta$$ is true. Specifically we can restrict its input to be $V^{ \ot n}$ and its output to be in $V^{\ot (M-1)} \bt V,$ in which case the formula becomes:
$$\Delta^{M-1,1}_M B^M_m = (\Id^{ \ot (M-1)} \bt B^1_{n+1-M} + B^{M-1}_{n-1} \bt \Id ) \Delta^{n-1,1}_n.$$
By the induction hypothesis we know that $B^{M-1}_{n-1} = \sum_{i+j=M-2} \Id^{\ot i} \ot b_{n+1-M} \ot \Id^{\ot j}$
so as an $M$-tensor the right hand side has the desired form; it's the formula given for $B^M_n.$
\end{proof}
\begin{proposition}\label{map}
A map of vector spaces $f:\overline{T}_c V \to W$ can uniquely be lifted to a morphism of coalgebras $F: \overline{T}_c V \to \overline{T}_c W,$ such that $pr_1 \circ F = f,$ when $pr_1$ is the natural projection $\overline{T}_c W \to W$. If we let $F^m_n$ be a composition; $$ F^m_n :V^{\ot n} \hookrightarrow  \overline{T}_c V \overset{F}{\to } \overline{T}_c W \to W^{\ot m}$$ and let $f_k=f|_{V^{\ot k}}$ then, explicitly, $F^m_n$ will be of the form 
$$
F^m_n =\begin{cases} 0 \mbox{ if } n<m \\ \sum_{i_1 + \dots + i_m=n} f_{i_1} \ot \ldots \ot f_{i_m} \end{cases}
$$ 
where $F$ can be recovered as the product; $$F = \prod_{m \geq 1} F^m \qquad F^m= \prod_{n \geq 1} F^m_n.$$
\end{proposition}
\begin{proof}
 By the property of the projection, $pr_1 \circ F = f,$ it follows that $F^1_n = f_n.$ We proceed by induction; assume that $$F^{m}_n = \sum_{ i_1 + \dots + i_{m}=n} f_{i_1} \ot \ldots \ot f_{i_m} ,$$ for all $m <M.$ The equation $$(F\bigotimes F) \circ \Delta = \Delta \circ F$$ can be restricted to taking the input $V^{\ot n}$ and having the output $W^{\ot (M-1)} \bigotimes W,$ in which case it becomes $$\Delta_M^{M-1,1} F^M_n = \sum_{i+j=n} (F^{M-1}_i \bt F^1_j) \Delta^{i,j}_n.$$ Now we can expand $F^{M-1}_i$ with the induction hypothesis and compare the two sides of the equation as an $M$-tensor. It follows that $F^M_n$ is given by the formula in the description of the theorem.
\end{proof}

\begin{proposition}\label{homotopy}
Let $F,G: \overline{T}_cV \to \overline{T}_c W$ be two morphisms of coalgebras and let $h: \overline{T}_c V \to W$ be a map of vector spaces then, there exist a unique map $H: \overline{T}_c V \to \overline{T}_c W,$ such that $(H\bigotimes G + F \bigotimes H) \Delta = \Delta H$ and so that $pr_1 \circ H = h$ when $pr_1$ is the natural projection $\overline{T}_c W \to W.$ Define $H^m_n$ as the composition $$ H^m_n :V^{\ot n} \hookrightarrow  \overline{T}_c V \overset{H} {\to} \overline{T}_c W \to W^{\ot m}$$  Explicitly $H^m_n$ is of the following form $$H^m_n = \begin{cases} 0 \mbox{ if } m>n \\ \sum_{a+b=m-1} \sum_{i_1+\ldots + i_a+s+j_1+\ldots+j_{b}=n} F^1_{i_1} \ot \ldots \ot F^1_{i_a} \ot h_{s} \ot G^1_{j_1} \ot \ldots \ot G^1_{j_b}, \end{cases}$$ where $a,b \geq 0,$ $s>0$ and, $F^1_k$ and $G^1_l$ are as in the previous theorem, and $h_i:=h|_{V^{\ot i}}.$ From $H^m_n$ we can recover $H$ by taking the product; $$H = \prod_{m \geq 1} H^m  \qquad H^m = \prod_{n \geq 1} H^m_n.$$
\end{proposition}
\begin{proof} We prove this with induction. When $m=1$ this follows from the projection property; $H^1_n=h_n.$ Assume that $$H^m_n = \sum_{a+b=m-1} \sum_{i_1+\ldots + i_a+s+j_1+\ldots+j_{b}=n} F^1_{i_1} \ot \ldots \ot F^1_{i_a} \ot h_{s} \ot G^1_{j_1} \ot \ldots \ot G^1_{j_b}$$ for all $m < M.$ 
Restrict the input to $V^{\ot n}$ and consider the projection to the $(M-1,1)$-th component in the equation $(H\bigotimes G + F \bigotimes H) \Delta = \Delta H$ to get:
\[
		 \Delta_M^{M-1,1} \circ H^N_n
		= \sum_{i+j=n} (F^{M-1}_i \bigotimes H^1_j + H^{M-1}_i \bigotimes G^{1}_j) \circ \Delta^{i,j}_n,
\]
if we expand $H^{M-1}_i$ with the induction hypothesis and $F^{M-1}_i$ with the previous theorem then we see that this, as an $M$-tensor, is precisely the formula that the theorem predicts.
\end{proof}

\begin{definition}
An $A_\infty$ algebra is a graded vector space $V$ equipped with the structure of a codifferential $b_V$ on the associated reduced tensor coalgebra (of the shifted vector space);  $$b_V:\overline{T_c} sV \to \overline{T_c} sV.$$

A morphism of $A_\infty$ algebras $$f:(V,b_V) \to (W, b_W)$$ is a morphism of dg coalgebras $$F: \overline{T_c} sV \to \overline{T_c} sW.$$

Let $F$ and $G$ be the morphism of dg coalgebras $$F,G: (\overline{T_c} sV,b_V) \to (\overline{T_c} sW,b_W),$$ where $b_V$ and $b_W$ are codifferentials giving $V$ and $W$ the structure of $A_\infty$ algebras. We say that a map $$H: \overline{T_c} sV \to \overline{T_c} sW$$ is a homotopy of $F$ and $G$ if it satisfies two relations: \begin{enumerate}
\item $(F \ot H + H \ot G) \circ \Delta_V = \Delta_W \circ H$
\item $F-G = b_W \circ H + H \circ b_V.$
\end{enumerate} 
\end{definition}
 
Proposition \ref{coder}, \ref{map} and \ref{homotopy}  provides us with a way to reinterpret the definitions concerning $A_\infty$ algebras without referencing the tensor coalgebra explicitly.

\begin{theorem}
An $A_\infty$ algebra structure on the graded vector space $V$ is a sequence of maps $m_n : V^{\ot n} \to V$ of degree $2-n$ such that following equations are satisfied 
\[
\begin{split}
 m_1 \circ m_1 &=0 \\
- m_2 \circ (\Id \ot m_1) - m_2\circ (m_1 \ot \Id) + m_1 \circ m_2&= 0 \\
-m_2 \circ (m_2 \ot \Id) - m_2 \circ ( \Id \ot m_2)+m_3 \circ (m_1 \ot \Id^{\ot 2})&\hphantom{=} \\+m_3 \circ(\Id \ot m_1 \ot \Id) +  m_3 \circ (\Id^{\ot 2} \ot m_1)+m_1 \circ m_3 &= 0 \\
 & \hskip2mm\vdots \\
\sum_{s+j+t=n} (-1)^{s+jt}  m_{s+1+t} \circ (\Id^{\ot s} \ot m_j \ot \Id^{\ot t}) &=0 
\end{split}
\]
\end{theorem}
\begin{proof}
The proof is a matter of expanding the expression $b_V \circ b_V=0$ with proposition \ref{coder} and recognizing that $m_n = s^{-1} \circ {b_V}_1^n \circ s^{\ot n},$ where ${b_V}_1^n$ is the restriction of $b_V$ to $(sV)^{\ot n}$ followed by the projection onto $sV.$ The sign factor comes from applying the Koszul sign rule when shifts are reorganized.
\end{proof}

We will occasionally denote an $A_\infty$ algebra with the pair $(V,m^V),$ where  $m^V$ is the system of maps given in the above theorem.

\begin{theorem}
A morphism of $A_\infty$ algebras $f:(V,m^V)\to (W,m^W)$ is a collection of maps $f_n: V^{\ot n} \to W$ of degree $1-n$ such that 
\[
\sum_{r+s+t=n} (-1)^{r+st} f_{r+1+t} \circ (\Id^{\ot r} \ot m^V_s \ot \Id^{\ot t}) 
= \sum_{q=1}^{n} \sum_{i_1 + \ldots + i_q=n} (-1)^{p} m^W_q \circ (f_{i_1} \ot  \ldots \ot f_{i_q})
\]
where $p=(q-1)(i_1-1)+(q-2)(i_2-1)+\ldots+2(i_{q-2}-1)+(i_{q-1}-1).$
\end{theorem}
\begin{proof}
Let $F: (\overline{T}_c sV, B_V) \to (\overline{T}_c sW, B_W)$ be a coalgebra morphism. Explicitly $f_n$ is given as $s^{-1}\circ F^1_n \circ s^{\ot n},$ where $F^1_n$ is the restriction of $F$ to the $n$:th component followed by the projection to the first;
$F^1_n: (sV)^{\ot n} \to sW.$ 

We start with the equation $B_W \circ F = F \circ B_V.$ In it we restrict the input to $(sV)^{\ot n}$ and output to $sW$. The result is $$\sum_{i=1}^n (B_W)^1_i \circ F^i_n = \sum_{j=1}^n F^1_j \circ (B_V)^j_n.$$  By proposition \ref{coder} and \ref{map} $B^m_n = \sum_{i+j=m-1} \Id^{\ot i} \ot B^1_s \ot \Id^{\ot j}$ and $F^m_n= \sum_{n_1+\ldots+n_m=n} F^1_{n_1} \ot \ldots \ot F^1_{n_m}.$ Using these explicit formulas we arrive at the expression $$\sum_{i=1}^n (B_W)^1_i \circ (\sum_{n_1+\ldots+n_i=n} F^1_{n_1} \ot \ldots \ot F^1_{n_i}) = \sum_{i=1}^{n} F^1_{i} \circ ( \sum_{a+b =i-1} \Id^{\ot a} (B_V)^1_{n+1-i} \ot \Id^{b}) $$
\end{proof}

\begin{theorem}
Let $(V,\mu^V)$ and $(W,\mu^W)$ be two $A_\infty$ algebras and let $f,g: (V,\mu^V) \to (W,\mu^W)$ be two $A_\infty$ algebra morphisms given on the form of maps \[\begin{split} f_n:V^{\ot n} &\to W \\ g_n:V^{\ot n} &\to W\end{split}.\] A system of maps of graded vector spaces $h_n:V^{\ot n} \to W$ of degree $-n$ is a homotopy of $f$ and $g$ if
\[
\begin{split}
f_n - g_n &=   \sum_{m=1}^n \sum_{\begin{array}{c} \scriptstyle{k+l = m-1}  \\  \scriptstyle{i_1+\ldots+i_k+t+j_1+\ldots j_l=n}  \end{array}} (-1)^s \mu_m^W  \circ \left( f_{i_1} \ot \ldots \ot f_{i_k} \ot h_t \ot g_{j_1} \ot \ldots \ot g_{j_l} \right) \\ & + \sum_{i+j+k=n} (-1)^{ij+k} h_{i+1+k} \circ \left( \Id^{\ot i} \ot \mu_j^V \ot \Id^{\ot k} \right),
\\ &  s=l+\sum_{1\leq a \leq l} (1-j_a)(n-\sum_{b \geq a} j_b ) + t \sum_{1\leq a \leq k} i_a + \sum_{2 \leq a \leq k} (1-i_a)(\sum_{b < a} i_b)
\end{split}
\]
\end{theorem}
\begin{proof} The proof is along the lines of the previous theorems. Use proposition \ref{map}, \ref{homotopy} to lift $h$ to a map $H,$ $f$ to $F$ and $g$ to $G.$ In that setting you can apply the rule for homotopy, project the formula to the first component and lastly you recognize the sign that comes from the shifts.
\end{proof} 

\end{section}

\bigskip 

\begin{section}{The Configuration Spaces $Conf_n(\mathbb{R}),$ $C_n(\mathbb{R})$ and $\mathfrak C_n(\mathbb{R})$ }

\begin{subsection}{Families of uncompactified configuration spaces}

Given a set $A$ we define the configuration space $Conf_A(\R)$ as the set of injections of the set $A$ into the real line; $$Conf_A(\R) :=\{A \hookrightarrow \R \}.$$ We think of this as $|A|$ distinct real points labeled by the elements from $A.$ In the special case when $A=[n]$ we use the notation $$Conf_n (\R) := Conf_{[n]}(\R).$$ Sometimes we will consider the full set of maps $A \to \R,$ and for it we introduce the notation $$\widetilde{Conf_A}(\R) := \{A \to \R \}.$$ 

The set $Conf_n(\R)$ is a real oriented manifold of dimension $n.$ As a space $Conf_n(\mathbb{R})$ is the union of $n!$ connected components, all isomorphic to $$Conf_n^o(\mathbb{R}) := \{x_1 < x_2 < \ldots < x_n\}.$$ The orientation is given on $Conf_n^o(\mathbb{R})$ as the volume form $dx_1 \wedge dx_2 \wedge \ldots \wedge x_n.$ The group $S_n$ acts on $Conf_n(\R)$ by permuting the elements of $[n].$ We assume that the action of $S_n$ is orientation preserving on $Conf_n(\mathbb{R})$ and this fixes the orientation on all connected components of $Conf_n(\mathbb{R}).$ 

The 2-dimensional Lie group $G_{(2)} = \R^{+} \ltimes \R$ acts freely on $Conf_n(\R)$ via the action $$(x_1,\ldots,x_n)\times (\lambda,\nu) =(\lambda x_1 + \nu, \ldots, \lambda x_n+\nu ).$$ The quotient space from this action is a $(n-2)$-dimensional real oriented manifold. Where we choose to represent equivalence classes of this quotient space with elements of the form $(0=x_1 < x_2 \ldots < x_{n-1} < x_n = 1 ),$ when we do so we will use the notation $C_n^o(\R);$ $$C_n^o(\R) := \{ [(0=x_1 < x_2 \ldots < x_{n-1} < x_n = 1 )] \}.$$ The orientation is given by the form $dx_2 \wedge \ldots \wedge dx_{n-1}.$ We also define $C_n(\R) := S_n \times C_n^o(\R).$

Alternatively we can represent equivalence classes of $$Conf_n(\mathbb{R}) / G_{(2)}$$ with elements $p=(x_1,\ldots,x_n) \in Conf_n(R)$ subject to $$x_c(p) = \frac{1}{n} \sum x_i =0 \qquad  ||p||=\sqrt{\sum (x_i -x_c(p))^2} = 1.$$ We define $$C_n^{st}(\R) := \{p \in Conf_n(\R) | x_c(p)=0, ||p||=1\}$$ and also $$\widetilde{C_n^{st}}(\R) := \{p \in \widetilde{Conf_n}(\R) | x_c(p)=0, ||p||=1\}.$$

The 1-dimensional Lie group $G_{(1)} =  \R$ acts on $Conf_n(\R)$ by translation $$(p,\nu) \mapsto p+\nu,$$ and we denote the quotient space $$\mathfrak C_n(\R) := Conf_n(\R) / G_{(1)}.$$ 
Define $$\mathfrak C_n^{st}(\R) :=\{ p \in Conf_n(\R) | x_c(p)=0\} \subset Conf_n(\R).$$

We have three isomorphisms associated to these configuration spaces.
\begin{enumerate}
\item The space $C_n(\R)$ is naturally isomorphic to $C_n^{st}(\R)$ 
\item We have $$\Psi_n : \mathfrak C_n(\R) \cong  C_n^{st}(\R) \times (0,1)$$ given by $$p \mapsto \left( \frac{p-x_c(p)}{||p||} , \frac{||p||}{1+||p||} \right).$$
\item We have $$\Phi_n :Conf_n(\R) \cong C_n^{st}(\R) \times (0,1) \times (-1,1)$$ given by 
 $$p \mapsto \left( \frac{p-x_c(p)}{||p||} , \frac{||p||}{1+||p||}, \frac{x_c(p)}{1+|x_c(p)|} \right).$$
\end{enumerate} 

These isomorphisms open up the door for compactifications of these configuration spaces which we discuss next. 
\end{subsection}

\begin{subsection}{A compactification of $C_n(\R)$}
We introduce a topological compactification $\overline{C_n}(\R)$ as the closure of the following injections 
\[ \xymatrix{  C_n(\R) \ar@{->}[r]^{\prod \pi_A \phantom{asdf}}   & \displaystyle \prod_{|A| \subset [n], |A| \geq 2} C_A(\R) \ar@{->}[r]^{\simeq} &  \displaystyle \prod_{|A| \subset [n], |A| \geq 2} C_A^{st} (\R) \ar@{^{(}->}[r]^{}   &  \displaystyle \prod_{|A| \subset [n], |A| \geq 2} \widetilde{C}_A^{st}(\R) }\]

The codimension one boundary strata of the configurations space $\overline{C}_n(\R)$ is given by \begin{itemize}
\item $$\partial \overline {C}_n(\R) = \bigcup_{A \subset [n]} \overline {C}_{n- |A|+1}(\R) \times \overline{C}_{|A|}(\R),$$ where $A$ is a connected proper subset of $[n]$ with two or more elements.
\item The face complex on $\overline{C}_\bullet (\R)$ has the natural structure of a dg free operad;
$$ \mathcal{F}ree\left \langle \xy
(1,-5)*{\ldots},
(-13,-7)*{_{i_1}},
(-8,-7)*{_{i_2}},
(-3,-7)*{_{i_3}},
(7,-7)*{_{i_{q-1}}},
(13,-7)*{_{i_q}},
 (0,0)*{\bullet}="a",
(0,5)*{}="0",
(-12,-5)*{}="b_1",
(-8,-5)*{}="b_2",
(-3,-5)*{}="b_3",
(8,-5)*{}="b_4",
(12,-5)*{}="b_5",
\ar @{-} "a";"0" <0pt>
\ar @{-} "a";"b_2" <0pt>
\ar @{-} "a";"b_3" <0pt>
\ar @{-} "a";"b_1" <0pt>
\ar @{-} "a";"b_4" <0pt>
\ar @{-} "a";"b_5" <0pt>
\endxy \right \rangle_{q \geq 2}$$
where the differential acts as follows 
\[
\partial \left( \xy
(1,-5)*{\ldots},
(-13,-7)*{_{i_1}},
(-8,-7)*{_{i_2}},
(-3,-7)*{_{i_3}},
(7,-7)*{_{i_{q-1}}},
(13,-7)*{_{i_q}},
 (0,0)*{\bullet}="a",
(0,5)*{}="0",
(-12,-5)*{}="b_1",
(-8,-5)*{}="b_2",
(-3,-5)*{}="b_3",
(8,-5)*{}="b_4",
(12,-5)*{}="b_5",
\ar @{-} "a";"0" <0pt>
\ar @{-} "a";"b_2" <0pt>
\ar @{-} "a";"b_3" <0pt>
\ar @{-} "a";"b_1" <0pt>
\ar @{-} "a";"b_4" <0pt>
\ar @{-} "a";"b_5" <0pt>
\endxy \right) = \sum_{k=1}^{q-2} \sum_{l=2}^{n-k} \epsilon(k,l) \begin{xy}
<0mm,0mm>*{\bullet},
<0mm,0.8mm>*{};<0mm,5mm>*{}**@{-},
<-9mm,-5mm>*{\ldots},
<-9mm,-7mm>*{_{i_1\ \,   \ \ \  \, i_k}},
<13mm,-7mm>*{_{i_{k+l+1}\ \, \ \ \  \, i_q}},
<0mm,-10mm>*{...},
<14mm,-5mm>*{\ldots},
<-0.7mm,-0.3mm>*{};<-13mm,-5mm>*{}**@{-},
<-0.6mm,-0.5mm>*{};<-6mm,-5mm>*{}**@{-},
<0.6mm,-0.3mm>*{};<20mm,-5mm>*{}**@{-},
<0.3mm,-0.5mm>*{};<8mm,-5mm>*{}**@{-},
<0mm,-0.5mm>*{};<0mm,-4.3mm>*{}**@{-},
<0mm,-5mm>*{\bullet};
<-5mm,-10mm>*{}**@{-},
<-2.7mm,-10mm>*{}**@{-},
<2.7mm,-10mm>*{}**@{-},
<5mm,-10mm>*{}**@{-},
<2mm,-12mm>*{_{i_{k+1}}\ \  \ \ _{i_{k+l}}},
\end{xy} 
\]
Where the factor $\epsilon(k,l)$ is a sign that can be worked out to be $(-1)^{k+l(n-k-l)+1}$.
Representations of this operad in differential graded vector space are given by $A_\infty$ structures. Thus this is a description of the $\mathcal A_\infty$ operad.
\end{itemize}

\begin{subsection}{The space $\overline{C}_n(\R)$ as a smooth manifold with corners}
Let $RT_{n,l} $ be the set of rooted trees with $n$ legs and $l+1$ internal vertices. The set $RT_{n,l}$ parametrizes the codimension $l$ boundary strata of $\overline{C}_n(\R)$ in the following sense. Each tree $t\in RT_{n,l}$ describes a space $C_t(\R)$ which is defined as the product
$$C_t(\R) :=\prod_{v \in vert(t)} C_{|in(v)|}(\R),$$ where, like before, $vert(t)$ denote the set of internal vertices of $t$ and $in(v)$ the set of input edges at the vertex $v.$ From this one gets a description of $\overline{C}_n(\R)$ as a stratified disjoint union of spaces
$$\overline{C}_n(\R) = \coprod_{l\geq 0} \prod_{t \in RT_{n,l}} C_t(\R).$$

To make the compactified configuration space $\overline{C}_n(\R)$ into a smooth manifold with corner we shall define coordinate charts $U_t$ near the boundary stratum $C_t(\R).$ We do this for a specific tree $t$ but the general procedure should be clear from the given example. Let $t$ be the tree 
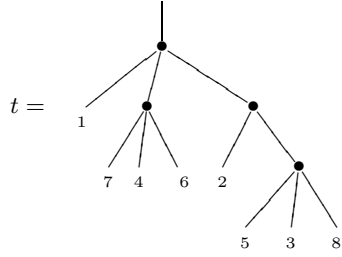
\[
t= \quad \xy
(-10.5,-2)*{_1},
(17,-18)*{_3},
(11,-18)*{_5},
(23,-18)*{_8},
(3,-10)*{_6},
(8,-10)*{_2},
(-3,-10)*{_4},
(-7,-10)*{_7},
(0,14)*{}="0",
 (0,8)*{\bullet}="a",
(-10,0)*{}="b_1",
(-2,0)*{\bullet}="b_2",
(12,0)*{\bullet}="b_3",
(2,-8)*{}="c_1",
(18,-8)*{\bullet}="n",
(-7,-8)*{}="c_2",
(8,-8)*{}="c_3",
(-3,-8)*{}="c_4",
(11,-16)*{}="d_1",
(17,-16)*{}="d_2",
(23,-16)*{}="d_3",
\ar @{-} "a";"0" <0pt>
\ar @{-} "a";"b_1" <0pt>
\ar @{-} "a";"b_2" <0pt>
\ar @{-} "a";"b_3" <0pt>
\ar @{-} "b_2";"c_1" <0pt>
\ar @{-} "b_2";"c_2" <0pt>
\ar @{-} "b_3";"c_3" <0pt>
\ar @{-} "b_2";"c_4" <0pt>
\ar @{-} "b_3";"n" <0pt>
\ar @{-} "n";"d_1" <0pt>
\ar @{-} "n";"d_2" <0pt>
\ar @{-} "n";"d_3" <0pt>
\endxy
\]
We define the coordinate chart close to $C_t(\R)$ in a three step procedure.
\begin{enumerate}
\item Associate to the tree $t$ a {\em metric tree}, $t_{metric}$ by endowing each internal edge with a bounded non-negative parameter $\epsilon;$
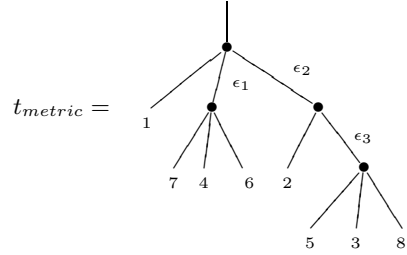
\[
t_{metric}= \quad \xy
(-10.5,-2)*{_1},
(17,-18)*{_3},
(11,-18)*{_5},
(23,-18)*{_8},
(3,-10)*{_6},
(8,-10)*{_2},
(-3,-10)*{_4},
(-7,-10)*{_7},
(2,3)*{_{\epsilon_1}},
(10,5)*{_{\epsilon_2}},
(18,-4)*{_{\epsilon_3}},
(0,14)*{}="0",
 (0,8)*{\bullet}="a",
(-10,0)*{}="b_1",
(-2,0)*{\bullet}="b_2",
(12,0)*{\bullet}="b_3",
(2,-8)*{}="c_1",
(18,-8)*{\bullet}="n",
(-7,-8)*{}="c_2",
(8,-8)*{}="c_3",
(-3,-8)*{}="c_4",
(11,-16)*{}="d_1",
(17,-16)*{}="d_2",
(23,-16)*{}="d_3",
\ar @{-} "a";"0" <0pt>
\ar @{-} "a";"b_1" <0pt>
\ar @{-} "a";"b_2" <0pt>
\ar @{-} "a";"b_3" <0pt>
\ar @{-} "b_2";"c_1" <0pt>
\ar @{-} "b_2";"c_2" <0pt>
\ar @{-} "b_3";"c_3" <0pt>
\ar @{-} "b_2";"c_4" <0pt>
\ar @{-} "b_3";"n" <0pt>
\ar @{-} "n";"d_1" <0pt>
\ar @{-} "n";"d_2" <0pt>
\ar @{-} "n";"d_3" <0pt>
\endxy
\]  with $\epsilon_1,\epsilon_2,\epsilon_3 \in [0,\epsilon).$
\item Pick an $S_n$-equivariant section $\gamma: C_n(\R) \to Conf_n(\R)$, of the natural projection $Conf_n(\R) \to C_n(\R)$ and associate to the image of $\gamma$ a smooth structure. The section could be either of the two description of $C_n(\R)$ we mentioned above; $C^{st}_n(\R)$ or the space of configurations where $x_1=0$ and $x_n=1.$
\item The coordinate chart $U_t$ can now be seen to be isomorphic to the smooth manifold with corners $[0,\epsilon)^{|E(t)|} \times \prod_{v \in vert(t)} C_{|in(v)}(\R).$ The isomorphism is given by the map $\Phi_t,$
$$\Phi_t:[0,\epsilon)^{|E(t)|} \times \prod_{v \in vert(t)} C_{|in(v)|}(\R) {\longrightarrow} U_t.$$
which we describe in the example of our tree $t.$ Coordinatewise it is defined as follows
\[
\begin{array}{ccccccccccc}
 (0,\epsilon)^3 & \hspace{-2mm} \times \hspace{-2mm}& C^{st}_3(\R) & \hspace{-2mm} \times\hspace{-2mm} & C^{st}_3(\R) & \hspace{-2mm} \times \hspace{-2mm} &  C^{st}_2(\R)  &\hspace{-2mm} \times\hspace{-2mm} & C^{st}_3(\R)
& {\longrightarrow} & C_8(\R) \\
 (\epsilon_1,\epsilon_2,\epsilon_3) & \hspace{-2mm} \times \hspace{-2mm} & (x_1, x',x'') & \hspace{-2mm}\times
 \hspace{-2mm}& (x_7, x_4,x_6) &\hspace{-2mm} \times \hspace{-2mm}&
(x_2,x''') &\hspace{-2mm} \times\hspace{-2mm} & (x_5,x_3,x_8) &\longrightarrow & (y_1, y_7, y_4,y_6, y_2,y_5, y_3,y_8)
\end{array}
\] 
according to $$\begin{array}{cccc}
y_1=&x_1 & y_2=&x''+\epsilon_2 x_2 \\
y_3=&x'' +\epsilon_2(x'''+\epsilon_3 x_3) & y_4 =& x'+\epsilon_1 x_4 \\
y_5=&x'' +\epsilon_2(x'''+\epsilon_3 x_5) & y_6 =& x'+\epsilon_1 x_6 \\
y_7=&x'+\epsilon_1 x_4 & y_5=&x'' +\epsilon_2(x'''+\epsilon_3 x_5) 
\end{array}$$
In general the map $\Phi_t$ is given as the recursive $\epsilon$-magnified substitution scheme.  If the coordinates $x_i, \ldots, x_{i+k}$ lie in a corolla controlled by the internal edge associated to the coordinate $x'$ and where the internal edge is parametrized by the factor $\epsilon,$ then the substitution give the new coordinates $x'+\epsilon x_i, \ldots, x'+\epsilon x_{i+k}.$ 
\end{enumerate}
\end{subsection}

\end{subsection}

\begin{subsection}{A compactification of $\widehat{\mathfrak C}_n(\R)$} \label{confmorass}
Define the compactification of $\mathfrak C_n(\R)$ as the closure of the following inclusions
 \[ \xymatrix{
      \mathfrak{C}_n(\R) \ar@{->}[r]^{\prod \pi_A \phantom{asdf}}   & \displaystyle \prod_{|A| \subset [n], |A| \geq 1} \mathfrak C_A(\R) \ar@{->}[r]^{\prod \Phi_A \phantom{a} } &  \displaystyle \prod_{|A| \subset [n], |A| \geq 1} C_A^{st} (\R) \times (0,1) \ar@{^{(}->}[r]^{}   &  \displaystyle \prod_{|A| \subset [n], |A| \geq 1} \widetilde{ C}_A^{st} (\R) \times [0,1]
}  \] 

The codimension one boundary strata of the configurations space $\widehat{\mathfrak{C}}_n(\R)$ is given by \begin{itemize}
\item $$\partial \widehat{\mathfrak{C}}_n(\R) = \bigcup \widehat{\mathfrak{C}}_{n-|A|+1}(\R) \times \overline{C}_{|A|}(\R) \cup \bigcup \overline{C}_k(\R) \times \widehat{\mathfrak{C}}_{|A_1|}(\R) \times \ldots \times \widehat{\mathfrak{C}}_{|A_k|}(\R)$$
 where $A$ is as above and where the $A_i$ are connected disjoint subsets of $[n]$ such that $inf A_1 < \ldots < inf A_k$ and $\cup A_i = [n].$ 
\item The face complex of the disjoint union $$\overline{C}_\bullet(\R) \sqcup \widehat{\mathfrak{C}}_{\bullet}(\R) \sqcup \overline{C}_\bullet(\R)$$ has the natural structure of a dg free operad of transformation type;
$$ \mathcal Free \left \langle \xy
(1,-5)*{\ldots},
(-13,-7)*{_{i_1}},
(-8,-7)*{_{i_2}},
(-3,-7)*{_{i_3}},
(7,-7)*{_{i_{q-1}}},
(13,-7)*{_{i_q}},
 (0,0)*{\bullet}="a",
(0,5)*{}="0",
(-12,-5)*{}="b_1",
(-8,-5)*{}="b_2",
(-3,-5)*{}="b_3",
(8,-5)*{}="b_4",
(12,-5)*{}="b_5",
\ar @{-} "a";"0" <0pt>
\ar @{-} "a";"b_2" <0pt>
\ar @{-} "a";"b_3" <0pt>
\ar @{-} "a";"b_1" <0pt>
\ar @{-} "a";"b_4" <0pt>
\ar @{-} "a";"b_5" <0pt>
\endxy , \xy
(1,-5)*{\ldots},
(-13,-7)*{_{i_1}},
(-8,-7)*{_{i_2}},
(-3,-7)*{_{i_3}},
(7,-7)*{_{i_{n-1}}},
(13,-7)*{_{i_n}},
 (0,0)*{\blacksquare}="a",
(0,5)*{}="0",
(-12,-5)*{}="b_1",
(-8,-5)*{}="b_2",
(-3,-5)*{}="b_3",
(8,-5)*{}="b_4",
(12,-5)*{}="b_5",
\ar @{.} "a";"0" <0pt>
\ar @{-} "a";"b_2" <0pt>
\ar @{-} "a";"b_3" <0pt>
\ar @{-} "a";"b_1" <0pt>
\ar @{-} "a";"b_4" <0pt>
\ar @{-} "a";"b_5" <0pt>
\endxy
,
\xy
(1,-5)*{\ldots},
(-13,-7)*{_{i_1}},
(-8,-7)*{_{i_2}},
(-3,-7)*{_{i_3}},
(7,-7)*{_{i_{n-1}}},
(13,-7)*{_{i_p}},
 (0,0)*{\circ}="a",
(0,5)*{}="0",
(-12,-5)*{}="b_1",
(-8,-5)*{}="b_2",
(-3,-5)*{}="b_3",  
(8,-5)*{}="b_4",
(12,-5)*{}="b_5",
\ar @{.} "a";"0" <0pt>
\ar @{.} "a";"b_2" <0pt>
\ar @{.} "a";"b_3" <0pt>
\ar @{.} "a";"b_1" <0pt>
\ar @{.} "a";"b_4" <0pt>
\ar @{.} "a";"b_5" <0pt>
\endxy  \right \rangle_{p,q \geq 2 \atop n\geq 1}$$
The differential has the following action 
\[
\begin{split}
& \partial \left( \xy
(1,-5)*{\ldots},
(-13,-7)*{_{i_1}},
(-8,-7)*{_{i_2}},
(-3,-7)*{_{i_3}},
(7,-7)*{_{i_{n-1}}},
(13,-7)*{_{i_n}},
 (0,0)*{\blacksquare}="a",
(0,5)*{}="0",
(-12,-5)*{}="b_1",
(-8,-5)*{}="b_2",
(-3,-5)*{}="b_3",
(8,-5)*{}="b_4",
(12,-5)*{}="b_5",
\ar @{.} "a";"0" <0pt>
\ar @{-} "a";"b_2" <0pt>
\ar @{-} "a";"b_3" <0pt>
\ar @{-} "a";"b_1" <0pt>
\ar @{-} "a";"b_4" <0pt>
\ar @{-} "a";"b_5" <0pt>
\endxy \right)=  
\sum_{l=1}^{n-1} \sum_{k=1}^{n-l} \epsilon(l,k) \begin{xy}
<0mm,0mm>*{\blacksquare},
<0mm,0.8mm>*{};<0mm,5mm>*{}**@{.},
<-9mm,-5mm>*{\ldots},
<-9mm,-7mm>*{_{i_1\ \,   \ \ \  \, i_k}},
<13mm,-7mm>*{_{i_{k+l+1}\ \, \ \ \  \, i_n}},
<0mm,-10mm>*{...},
<14mm,-5mm>*{\ldots},
<-0.7mm,-0.3mm>*{};<-13mm,-5mm>*{}**@{-},
<-0.6mm,-0.5mm>*{};<-6mm,-5mm>*{}**@{-},
<0.6mm,-0.3mm>*{};<20mm,-5mm>*{}**@{-},
<0.3mm,-0.5mm>*{};<8mm,-5mm>*{}**@{-},
<0mm,-0.5mm>*{};<0mm,-4.3mm>*{}**@{-},
<0mm,-5mm>*{\bullet};
<-5mm,-10mm>*{}**@{-},
<-2.7mm,-10mm>*{}**@{-},
<2.7mm,-10mm>*{}**@{-},
<5mm,-10mm>*{}**@{-},
<2mm,-12mm>*{_{i_{k+1}}\ \  \ \ _{i_{k+l}}},
\end{xy} 
\\ & + \sum_{k=2}^n \sum_{n=n_1+\ldots+ n_k}
\epsilon(k;n_1,\ldots,n_k)\xy
(-15.5,-7)*{...},
(19,-7)*{...},
(7.5,0)*{\ldots},
(-16,-9)*{_{i_1\ \ \ldots\ \ i_{n_1}}},
(1,-9)*{_{i_{n_1+1}\ldots \ i_{n_1+n_2}}},
(20,-9)*{\ldots\ \ _{i_{n}}},
(-1.8,-7)*{...},
(0,7)*{\circ}="a",
(-14,0)*{\blacksquare}="b_0",
(-4.5,0)*{\blacksquare}="b_2",
(15,0)*{\blacksquare}="b_3",
(0,13)*{}="0",
(1,-7)*{}="c_1",
(-8,-7)*{}="c_2",
(-5,-7)*{}="c_3",
(-22,-7)*{}="d_1",
(-19,-7)*{}="d_2",
(-13,-7)*{}="d_3",
(12,-7)*{}="e_1",
(15,-7)*{}="e_2",
(22,-7)*{}="e_3",
\ar @{.} "a";"0" <0pt>
\ar @{.} "a";"b_0" <0pt>
\ar @{.} "a";"b_2" <0pt>
\ar @{.} "a";"b_3" <0pt>
\ar @{-} "b_2";"c_1" <0pt>
\ar @{-} "b_2";"c_2" <0pt>
\ar @{-} "b_2";"c_3" <0pt>
\ar @{-} "b_0";"d_1" <0pt>
\ar @{-} "b_0";"d_2" <0pt>
\ar @{-} "b_0";"d_3" <0pt>
\ar @{-} "b_3";"e_1" <0pt>
\ar @{-} "b_3";"e_2" <0pt>
\ar @{-} "b_3";"e_3" <0pt>
\endxy
\end{split}
\]
Where $\epsilon(k,l)=(-1)^{k+l+l(n-k)+1}$ and $$\epsilon(k;n_1,\ldots,n_k)=(-1)^{(k-1)(n_1-1)+(k-2)(n_2-1)+ \ldots + 2(n_{k-2}-1)+n_{k-1}-1}$$ 
On the corollas corresponding to the $A_\infty$ structure, \[\xy
(1,-5)*{\ldots},
(-13,-7)*{_{i_1}},
(-8,-7)*{_{i_2}},
(-3,-7)*{_{i_3}},
(7,-7)*{_{i_{q-1}}},
(13,-7)*{_{i_q}},
 (0,0)*{\bullet}="a",
(0,5)*{}="0",
(-12,-5)*{}="b_1",
(-8,-5)*{}="b_2",
(-3,-5)*{}="b_3",
(8,-5)*{}="b_4",
(12,-5)*{}="b_5",
\ar @{-} "a";"0" <0pt>
\ar @{-} "a";"b_2" <0pt>
\ar @{-} "a";"b_3" <0pt>
\ar @{-} "a";"b_1" <0pt>
\ar @{-} "a";"b_4" <0pt>
\ar @{-} "a";"b_5" <0pt>
\endxy \text{ and }
\xy
(1,-5)*{\ldots},
(-13,-7)*{_{i_1}},
(-8,-7)*{_{i_2}},
(-3,-7)*{_{i_3}},
(7,-7)*{_{i_{n-1}}},
(13,-7)*{_{i_p}},
 (0,0)*{\circ}="a",
(0,5)*{}="0",
(-12,-5)*{}="b_1",
(-8,-5)*{}="b_2",
(-3,-5)*{}="b_3",  
(8,-5)*{}="b_4",
(12,-5)*{}="b_5",
\ar @{.} "a";"0" <0pt>
\ar @{.} "a";"b_2" <0pt>
\ar @{.} "a";"b_3" <0pt>
\ar @{.} "a";"b_1" <0pt>
\ar @{.} "a";"b_4" <0pt>
\ar @{.} "a";"b_5" <0pt>
\endxy,\]
the differential acts precisely like in the case $\overline{C}_n(\R).$
Representations of this operad are given by a three pieces of data: two $A_\infty$ algebras, $A$ and $A',$ and a morphism of $A_\infty$ algebras $A \to A'.$ Thus this is the previously discussed operad $\mathcal Mor(As)_\infty.$
\end{itemize}
\end{subsection}
\begin{subsection} {The space $\widehat{\mathfrak{C}}_n(\R)$ as a smooth manifold with corners}
We generalize the procedure for $\overline{C}_n(\R)$ to $\widehat{\mathfrak{C}}_n(\R).$ For every tree $t \in \mathcal Mor ( \mathcal A_\infty)$ we define the sets $vert_{\bullet,\circ}(t)$ and $vert_{\blacksquare}(t)$ as the vertices of $t$ marked by $\{\bullet,\circ\}$ or $\blacksquare$. For the tree $t$ we define $\mathfrak C_t(\R)$ as a product; $$\mathfrak C_t(\R) := \prod_{v \in vert_{\bullet,\circ} (t)} C_{|in(v)|}(\R)  \times \prod_{v \in vert_\blacksquare(t)} \mathfrak C_{|in(v)|} (\R).$$
We can describe the space $\widehat{\mathfrak{C}}_n(\R)$ as a stratified union of spaces; $$\widehat{\mathfrak{C}}_n(\R) = \prod_{t \in \mathcal Mor(As)_\infty(n)} \mathfrak C_t(\R).$$
We shall define a coordinate chart $U_t$ around every boundary stratum $\mathfrak C_t(\R)$ with a metric tree. We associate to $t$ the metric tree $t_{metric}$ with for
\begin{enumerate}
\item  every internal edge of the type $ \xy (0,3)*{\blacksquare}="u", (0,-3)*{\bullet}="n", \ar @{-} "u";"n" <0pt> \endxy $ a small positive parameter $\epsilon;$
\item  every vertex of a dashed corolla associate a large positive number $\tau,$ 
\[\xy (0,0)*{\circ}="o", (2,1)*{_\tau}, (0,5)*{}="u", (-10,-5)*{}="l_1", (-5,-5)*{}="l_2", (1,-4)*{\ldots}="l_3", (10,-5)*{}="l_4" \ar @{.}, "o";"u" <0pt> \ar @{.} "o";"l_1" <0pt> \ar @{.} "o";"l_2" <0pt> \ar @{.} "o";"l_4" <0pt>,  \endxy\]
\item  every subgraph of $t_{metric}$ of the type $ \xy (0,3)*{\circ}="u", (0,-3)*{\circ}="n",(2,4)*{\tau_1}, (2,-2)*{\tau_2}, \ar @{.} "u";"n"<0pt> \endxy$ an inequality $\tau_1 > \tau_2.$
\end{enumerate}
\begin{example} As an example we consider a specific tree. The general method should be clear from this description. Let $t$ be the following tree 
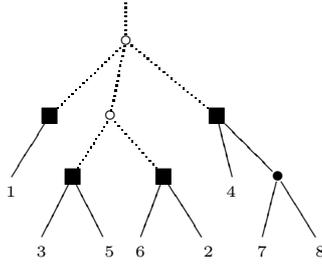
\[
\xy
(-15,-10)*{_1},
(-11,-18)*{_3},
(-2,-18)*{_5},
(2,-18)*{_6},
(11,-18)*{_2},
(14,-10)*{_4},
(18,-18)*{_7},
(25.7,-18)*{_8},
(0,15)*{}="0",
 (0,10)*{\circ}="a",
(-10,0)*{\blacksquare}="b_1",
(-2,0)*{\circ}="b_2",
(12,0)*{\blacksquare}="b_3",
(-15,-8)*{}="c_0",
(5,-8)*{\blacksquare}="c_1",
(-7,-8)*{\blacksquare}="c_2",
(10,-16)*{}="c_3",
(14,-8)*{}="c_4",
(20,-8)*{\bullet}="c_5",
(-11,-16)*{}="d_1",
(-3,-16)*{}="d_2",
(2,-16)*{}="d_3",
(18,-16)*{}="d_4",
(25,-16)*{}="d_5",
\ar @{.} "a";"0" <0pt>
\ar @{.} "a";"b_1" <0pt>
\ar @{.} "a";"b_2" <0pt>
\ar @{.} "a";"b_3" <0pt>
\ar @{-} "b_1";"c_0" <0pt>
\ar @{.} "b_2";"c_1" <0pt>
\ar @{.} "b_2";"c_2" <0pt>
\ar @{-} "c_1";"c_3" <0pt>
\ar @{-} "b_3";"c_4" <0pt>
\ar @{-} "b_3";"c_5" <0pt>
\ar @{-} "c_2";"d_1" <0pt>
\ar @{-} "c_2";"d_2" <0pt>
\ar @{-} "c_1";"d_3" <0pt>
\ar @{-} "c_5";"d_4" <0pt>
\ar @{-} "c_5";"d_5" <0pt>
\endxy
\]
Then the associated metric tree, $t_{metric},$ is given by 
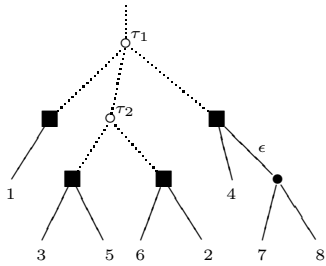
\[
\xy
(-15,-10)*{_1},
(-11,-18)*{_3},
(-2,-18)*{_5},
(2,-18)*{_6},
(11,-18)*{_2},
(14,-10)*{_4},
(18,-18)*{_7},
(25.7,-18)*{_8},
(0,15)*{}="0",
 (0,10)*{\circ}="a",
(2,11)*{_{\tau_1}},
(-10,0)*{\blacksquare}="b_1",
(-2,0)*{\circ}="b_2",
(0,1)*{_{\tau_2}},
(12,0)*{\blacksquare}="b_3",
(-15,-8)*{}="c_0",
(5,-8)*{\blacksquare}="c_1",
(-7,-8)*{\blacksquare}="c_2",
(10,-16)*{}="c_3",
(14,-8)*{}="c_4",
(20,-8)*{\bullet}="c_5",
(18,-4)*{_\epsilon},
(-11,-16)*{}="d_1",
(-3,-16)*{}="d_2",
(2,-16)*{}="d_3",
(18,-16)*{}="d_4",
(25,-16)*{}="d_5",
\ar @{.} "a";"0" <0pt>
\ar @{.} "a";"b_1" <0pt>
\ar @{.} "a";"b_2" <0pt>
\ar @{.} "a";"b_3" <0pt>
\ar @{-} "b_1";"c_0" <0pt>
\ar @{.} "b_2";"c_1" <0pt>
\ar @{.} "b_2";"c_2" <0pt>
\ar @{-} "c_1";"c_3" <0pt>
\ar @{-} "b_3";"c_4" <0pt>
\ar @{-} "b_3";"c_5" <0pt>
\ar @{-} "c_2";"d_1" <0pt>
\ar @{-} "c_2";"d_2" <0pt>
\ar @{-} "c_1";"d_3" <0pt>
\ar @{-} "c_5";"d_4" <0pt>
\ar @{-} "c_5";"d_5" <0pt>
\endxy
\]
\end{example}
Choose an equivariant section, $s:\mathfrak C_n(\R) \to Conf_n(\R)$ to the projection $Conf_n(\R) \to \mathfrak C_n(\R)$ and a smooth structure on the image of $s.$ Define $\mathfrak C_n^{st}(\R) := s(\mathfrak C_n(\R)),$ which is called the space of configurations in standard position. One possible choice of $\mathfrak C_n^{st}(\R)$ is subspace of points in $Conf_n(\R)$ where $\sum x_i =0.$  

The coordinate chart $U_t \subset \widehat{\mathfrak C}_n(\R)$ is now defined to be isomorphic to the manifold with corners,
$$(l,+\infty]^{|vert_{\circ}|(t)} \times [0,s)^{|edge_{\bullet}^{\blacksquare}(t)|} \times \prod_{v \in vert_{\circ,\bullet}(t)} C_{|in(v)|}^{st} (\R) \times \prod_{v \in vert_{\blacksquare}(t)} \mathfrak C_{|in(t)|}^{st}(\R)$$
where $vert_{\circ}$ denotes the set of vertices of type $\circ$, $vert_{\circ,\bullet}$ denotes the set of vertices of type $\circ$ or $\bullet$ and $edge_{\bullet}^{\blacksquare}$ denote the set of edges of type  $ \xy (0,3)*{\blacksquare}="u", (0,-3)*{\bullet}="n", \ar @{-} "u";"n" <0pt> \endxy.$ The isomorphism $\Phi_t$ between the coordinate chart $U_t$ and the product above is read from metric tree. We map is given in coordinates, for the specific tree in the above example, as follows
\[
\begin{array}{cccccccccccccccccccc}
(l,+\infty]^2 & \times & [0,s) & \times & C_{3}^{st} (\R) &\times &C_2^{st}(\R) &\times &C_2^{st}(\R)  \\ (\tau_1,\tau_2) & \times & \epsilon & \times & (x',x'',x''') & \times & (t',t'') & \times & (x_7,x_8) \\&\\& \times &  \mathfrak C_{1}^{st}(\R) &\times &\mathfrak C_{2}^{st}(\R)&\times &\mathfrak C_{2}^{st}(\R) &\times &\mathfrak C_{2}^{st}(\R)& \\ & \times & x_1 & \times & (x_3,x_5) & \times & (x_6,x_2) & \times & (x_4,u) \\ &&&&&&&&& \longrightarrow & \mathfrak C_8(\R) \\ &&&&& &&&&& (y_1, \ldots, y_8)
\end{array}
\]

such that 
\[
\begin{array}{cccc}
y_1=&\tau_1 x'+x_1 & y_2 =& \tau_1 x''+\tau_2 t'' + x_2 \\
y_3= &\tau_1 x''+\tau_2 t' + x_3  & y_4 =& \tau_1 x'''+x_4 \\
y_5=&\tau_1 x''+\tau_2 t' + x_5 & y_6 = &\tau_1 x''+\tau_2 t'' + x_6 \\
y_7=&\tau_1 x'''+u + \epsilon x_7&  y_8 =& \tau_1 x'''+u + \epsilon x_8   
\end{array}
\]

The boundary strata in $U_t$ are given by allowing formally $\tau_1=\infty, \tau_2 = \infty$ such that $\tau_1/\tau_2 =0$ and $\epsilon =0.$
\end{subsection}
\end{section}

\bigskip

\begin{section}{The operad $\mathcal Ho(As)_\infty$}
\begin{subsection}{Compactification of the Configuration Space $Conf_\bullet (\mathbb{R})$}
In this section we introduce our main result. We define the new compactification of the configuration space $Conf_n(\R)$ as the closure of the following injections
\[ \xymatrix{
      Conf_n(\R) \ar@{->}[r]^{\prod \pi_A \phantom{asdf}}   & \displaystyle \prod_{|A| \subset [n], |A| \geq 1}  Conf_A(\R) \ar@{->}[d]^{\prod \Phi_A \phantom{a} } &  \\ &  \displaystyle \prod_{|A| \subset [n], |A| \geq 1} C_A^{st} (\R) \times (0,1) \times (-1,1) \ar@{^{(}->}[r]^{}   &  \displaystyle \prod_{|A| \subset [n], |A| \geq 1} \widetilde{ C}_A^{st} (\R) \times [0,1] \times [-1,1] 
} 
 \]
\end{subsection}
We extend the previous result to the whole of $\overline{Conf}_n(\R).$
The codimension one boundary strata of $\overline{Conf}_n(\R)$ are given as 
\[
\begin{split}
& \partial \overline{Conf}_n(\R) =  \bigcup \overline{Conf}_{n-|A|+1}(\R) \times \overline{C}_{|A|} (\R) \cup \widehat{\mathfrak{C}}_{n}(\R) \cup \widehat{\mathfrak{C}}_{n}(\R) 
\\& \bigcup \overline{C}_{k+1+l}(\R) \times \widehat{\mathfrak{C}}_{|A_1|}(\R) \times \ldots \times \widehat{\mathfrak{C}}_{|A_k|}(\R) \times \overline{Conf}_{|A|}(\R) \times \widehat{\mathfrak{C}}_{|B_1|}(\R) \times \ldots \times \widehat{\mathfrak{C}}_{|B_l|}(\R) 
\end{split}
\]
\begin{enumerate}
\item The first union runs over all connected subsets $A \subset [n]$ such that $|A|>1.$ The stratum correspond to the collapsing of the points of $A$ into one point.
\item The stratum $\widehat{\mathfrak{C}}_{n}(\R)$ appears when either all points go to plus or minus infinity but in such a manner that the distance between the points is finite.
\item The second union runs over all partitions of $[n]$ into connected non-empty subsets $[n]=A_1\cup \ldots \cup A_k \cup F \cup B_1 \cup \ldots \cup B_l$ with $|F|>0.$ These limit points correspond to when the points from $A_1, \ldots, A_k$ go to $-\infty,$ the points from $F$ stay in a finite position and the points from $B_1,\ldots,B_l$ go to $\infty.$ The points do this such that each point in $A_i$ and $B_j$ remain a finite distance from each other; $||p_{A_i}||,||p_{B_j}||<\infty.$ 
\end{enumerate}
 By methods described in \cite{Me} we can consider the fundamental chains of $\{\overline{C}_{\bullet}(\R) \sqcup \widehat{\mathfrak{C}}_{\bullet}(\R) \sqcup \overline{Conf}_{\bullet} (\R) \sqcup \widehat{\mathfrak{C}}_{\bullet}(\R) \sqcup \overline{C}_{\bullet}(\R) \}$ as a dg free operad with two colors.
We identify the faces with corollas;
$$
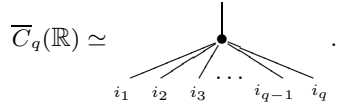
\overline{C}_q(\R) \simeq
\xy
(1,-5)*{\ldots},
(-13,-7)*{_{i_1}},
(-8,-7)*{_{i_2}},
(-3,-7)*{_{i_3}},
(7,-7)*{_{i_{q-1}}},
(13,-7)*{_{i_q}},
 (0,0)*{\bullet}="a",
(0,5)*{}="0",
(-12,-5)*{}="b_1",
(-8,-5)*{}="b_2",
(-3,-5)*{}="b_3",
(8,-5)*{}="b_4",
(12,-5)*{}="b_5",
\ar @{-} "a";"0" <0pt>
\ar @{-} "a";"b_2" <0pt>
\ar @{-} "a";"b_3" <0pt>
\ar @{-} "a";"b_1" <0pt>
\ar @{-} "a";"b_4" <0pt>
\ar @{-} "a";"b_5" <0pt>
\endxy.
$$
We need to illustrate two versions of this space as it appears either as collapsing or as controlling points at infinity. We distinguish between them by the color of their internal vertex and legs; drawn black or white/dashed.
$$
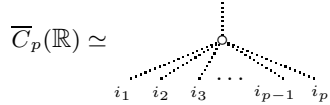
\overline{C}_p(\R) \simeq
\xy
(1,-5)*{\ldots},
(-13,-7)*{_{i_1}},
(-8,-7)*{_{i_2}},
(-3,-7)*{_{i_3}},
(7,-7)*{_{i_{p-1}}},
(13,-7)*{_{i_p}},
 (0,0)*{\circ}="a",
(0,5)*{}="0",
(-12,-5)*{}="b_1",
(-8,-5)*{}="b_2",
(-3,-5)*{}="b_3",
(8,-5)*{}="b_4",
(12,-5)*{}="b_5",
\ar @{.} "a";"0" <0pt>
\ar @{.} "a";"b_2" <0pt>
\ar @{.} "a";"b_3" <0pt>
\ar @{.} "a";"b_1" <0pt>
\ar @{.} "a";"b_4" <0pt>
\ar @{.} "a";"b_5" <0pt>
\endxy
$$
Points going to plus or minus infinity in a cluster are given a two-colored corolla: 
$$-\infty:
\widehat{\mathfrak{C}}_{n}(\R)\simeq
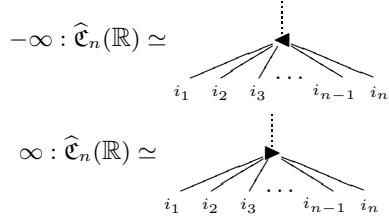
\xy
(1,-5)*{\ldots},
(-13,-7)*{_{i_1}},
(-8,-7)*{_{i_2}},
(-3,-7)*{_{i_3}},
(7,-7)*{_{i_{n-1}}},
(13,-7)*{_{i_n}},
 (0,0)*{\blacktriangleleft}="a",
(0,5)*{}="0",
(-12,-5)*{}="b_1",
(-8,-5)*{}="b_2",
(-3,-5)*{}="b_3",
(8,-5)*{}="b_4",
(12,-5)*{}="b_5",
\ar @{.} "a";"0" <0pt>
\ar @{-} "a";"b_2" <0pt>
\ar @{-} "a";"b_3" <0pt>
\ar @{-} "a";"b_1" <0pt>
\ar @{-} "a";"b_4" <0pt>
\ar @{-} "a";"b_5" <0pt>
\endxy
$$
$$\infty:
\widehat{\mathfrak{C}}_{n}(\R)\simeq
\xy
(1,-5)*{\ldots},
(-13,-7)*{_{i_1}},
(-8,-7)*{_{i_2}},
(-3,-7)*{_{i_3}},
(7,-7)*{_{i_{n-1}}},
(13,-7)*{_{i_n}},
 (0,0)*{\blacktriangleright}="a",
(0,5)*{}="0",
(-12,-5)*{}="b_1",
(-8,-5)*{}="b_2",
(-3,-5)*{}="b_3",  
(8,-5)*{}="b_4",
(12,-5)*{}="b_5",
\ar @{.} "a";"0" <0pt>
\ar @{-} "a";"b_2" <0pt>
\ar @{-} "a";"b_3" <0pt>
\ar @{-} "a";"b_1" <0pt>
\ar @{-} "a";"b_4" <0pt>
\ar @{-} "a";"b_5" <0pt>
\endxy
$$
We represent points staying finite with a two-colored corolla as follows: 
$$
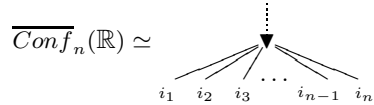
\overline{Conf}_{n}(\R)\simeq
\xy
(1,-5)*{\ldots},
(-13,-7)*{_{i_1}},
(-8,-7)*{_{i_2}},
(-3,-7)*{_{i_3}},
(7,-7)*{_{i_{n-1}}},
(13,-7)*{_{i_n}},
 (0,0)*{\blacktriangledown}="a",
(0,5)*{}="0",
(-12,-5)*{}="b_1",
(-8,-5)*{}="b_2",
(-3,-5)*{}="b_3",
(8,-5)*{}="b_4",
(12,-5)*{}="b_5",
\ar @{.} "a";"0" <0pt>
\ar @{-} "a";"b_2" <0pt>
\ar @{-} "a";"b_3" <0pt>
\ar @{-} "a";"b_1" <0pt>
\ar @{-} "a";"b_4" <0pt>
\ar @{-} "a";"b_5" <0pt>
\endxy
$$
In this graphical notation the differential has the following action:
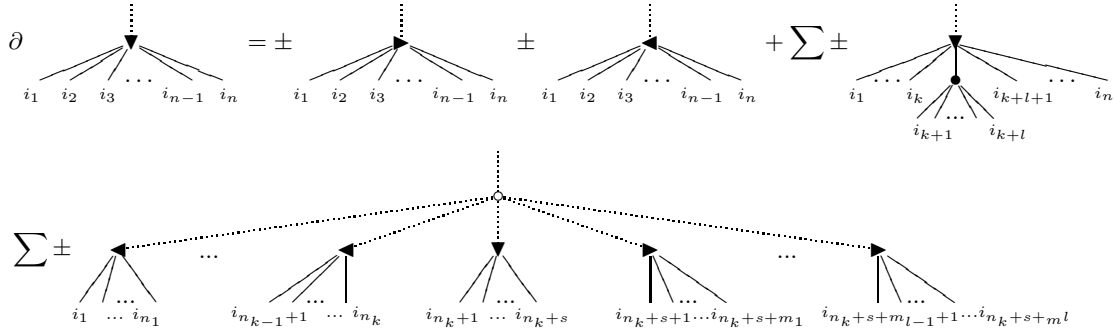
\[
\begin{split}
& \partial \xy
(1,-5)*{\ldots},
(-13,-7)*{_{i_1}},
(-8,-7)*{_{i_2}},
(-3,-7)*{_{i_3}},
(7,-7)*{_{i_{n-1}}},
(13,-7)*{_{i_n}},
 (0,0)*{\blacktriangledown}="a",
(0,5)*{}="0",
(-12,-5)*{}="b_1",
(-8,-5)*{}="b_2",
(-3,-5)*{}="b_3",
(8,-5)*{}="b_4",
(12,-5)*{}="b_5",
\ar @{.} "a";"0" <0pt>
\ar @{-} "a";"b_2" <0pt>
\ar @{-} "a";"b_3" <0pt>
\ar @{-} "a";"b_1" <0pt>
\ar @{-} "a";"b_4" <0pt>
\ar @{-} "a";"b_5" <0pt>
\endxy = \pm \xy
(1,-5)*{\ldots},
(-13,-7)*{_{i_1}},
(-8,-7)*{_{i_2}},
(-3,-7)*{_{i_3}},
(7,-7)*{_{i_{n-1}}},
(13,-7)*{_{i_n}},
 (0,0)*{\blacktriangleright}="a",
(0,5)*{}="0",
(-12,-5)*{}="b_1",
(-8,-5)*{}="b_2",
(-3,-5)*{}="b_3",
(8,-5)*{}="b_4",
(12,-5)*{}="b_5",
\ar @{.} "a";"0" <0pt>
\ar @{-} "a";"b_2" <0pt>
\ar @{-} "a";"b_3" <0pt>
\ar @{-} "a";"b_1" <0pt>
\ar @{-} "a";"b_4" <0pt>
\ar @{-} "a";"b_5" <0pt>
\endxy
\pm \xy
(1,-5)*{\ldots},
(-13,-7)*{_{i_1}},
(-8,-7)*{_{i_2}},
(-3,-7)*{_{i_3}},
(7,-7)*{_{i_{n-1}}},
(13,-7)*{_{i_n}},
 (0,0)*{\blacktriangleleft}="a",
(0,5)*{}="0",
(-12,-5)*{}="b_1",
(-8,-5)*{}="b_2",
(-3,-5)*{}="b_3",
(8,-5)*{}="b_4",
(12,-5)*{}="b_5",
\ar @{.} "a";"0" <0pt>
\ar @{-} "a";"b_2" <0pt>
\ar @{-} "a";"b_3" <0pt>
\ar @{-} "a";"b_1" <0pt>
\ar @{-} "a";"b_4" <0pt>
\ar @{-} "a";"b_5" <0pt>
\endxy
+ 
\sum \pm \begin{xy}
<0mm,0mm>*{\blacktriangledown},
<0mm,0.8mm>*{};<0mm,5mm>*{}**@{.},
<-9mm,-5mm>*{\ldots},
<-9mm,-7mm>*{_{i_1\ \,   \ \ \  \, i_k}},
<13mm,-7mm>*{_{i_{k+l+1}\ \, \ \ \  \, i_n}},
<0mm,-10mm>*{...},
<14mm,-5mm>*{\ldots},
<-0.7mm,-0.3mm>*{};<-13mm,-5mm>*{}**@{-},
<-0.6mm,-0.5mm>*{};<-6mm,-5mm>*{}**@{-},
<0.6mm,-0.3mm>*{};<20mm,-5mm>*{}**@{-},
<0.3mm,-0.5mm>*{};<8mm,-5mm>*{}**@{-},
<0mm,-0.5mm>*{};<0mm,-4.3mm>*{}**@{-},
<0mm,-5mm>*{\bullet};
<-5mm,-10mm>*{}**@{-},
<-2.7mm,-10mm>*{}**@{-},
<2.7mm,-10mm>*{}**@{-},
<5mm,-10mm>*{}**@{-},
<2mm,-12mm>*{_{i_{k+1}}\ \  \ \ _{i_{k+l}}},
\end{xy} 
\\ & \sum \pm
\begin{xy}
(-8,-1)*{...},
(68,-1)*{...},
(-19,-7)*{...},
(6,-7)*{...},
(31,-7)*{...},
(55,-7)*{...},
(85,-7)*{...},
(-20,-9)*{_{i_1\ \ \ldots  \ i_{n_1}}},
(5,-9)*{_{i_{n_{k-1}+1}\ \ \ldots  \ i_{n_k}}},
(30,-9)*{_{i_{n_k+1} \ \ldots  \ i_{n_k+s}}},
(58,-9)*{_{i_{n_k+s+1} \ldots   i_{n_{k}+s+m_1}}},
(89,-9)*{_{i_{n_k+s+m_{l-1}+1} \ldots   i_{n_k+s+_ml}}},
(30,7)*{\circ}="a",
(-20,0)*{\blacktriangleleft}="b_0",
(10,0)*{\blacktriangleleft}="b_3",
(30,0)*{\blacktriangledown}="b_4",
(50,0)*{\blacktriangleright}="b_5",
(80,0)*{\blacktriangleright}="b_7",
(30,13)*{}="0",
(-10,-7)*{}="c_1",
(0,-7)*{}="c_3",
(-7,-7)*{}="c_2",
(-25,-7)*{}="d_1",
(-22,-7)*{}="d_3",
(-15,-7)*{}="d_2",
(0,-7)*{}="e_1",
(3,-7)*{}="e_3",
(10,-7)*{}="e_2",
(25,-7)*{}="f_1",
(28,-7)*{}="f_3",
(35,-7)*{}="f_2",
(50,-7)*{}="g_1",
(53,-7)*{}="g_2",
(60,-7)*{}="g_3",
(60,-7)*{}="h_1",
(63,-7)*{}="h_2",
(70,-7)*{}="h_2",
(80,-7)*{}="i_1",
(83,-7)*{}="i_2",
(90,-7)*{}="i_3",
\ar @{.} "a";"0" <0pt>
\ar @{.} "a";"b_0" <0pt>
\ar @{.} "a";"b_3" <0pt>
\ar @{.} "a";"b_4" <0pt>
\ar @{.} "a";"b_5" <0pt>
\ar @{.} "a";"b_7" <0pt>
\ar @{-} "b_0";"d_1" <0pt>
\ar @{-} "b_0";"d_2" <0pt>
\ar @{-} "b_0";"d_3" <0pt>
\ar @{-} "b_3";"e_1" <0pt>
\ar @{-} "b_3";"e_2" <0pt>
\ar @{-} "b_3";"e_3" <0pt>
\ar @{-} "b_4";"f_1" <0pt>
\ar @{-} "b_4";"f_2" <0pt>
\ar @{-} "b_4";"f_3" <0pt>
\ar @{-} "b_5";"g_1" <0pt>
\ar @{-} "b_5";"g_2" <0pt>
\ar @{-} "b_5";"g_3" <0pt>
\ar @{-} "b_7";"i_1" <0pt>
\ar @{-} "b_7";"i_3" <0pt>
\ar @{-} "b_7";"i_2" <0pt>
\end{xy}
\end{split}
\]
On the corollas corresponding to $\overline{C}_p(\R)$ and $\widehat{\mathfrak{C}}_p(\R)$ the differential acts identically to the differential in the $\mathcal Mor(As)_\infty$ operad.

\begin{example}
We look at the simple example $Conf_2(\R):$
$$  Conf_2(R) \hookrightarrow \widetilde{C_{(1)}}^{st} (\R) \times [0,1] \times [-1,1]\times\widetilde{C_{(2)}}^{st} (\R) \times [0,1] \times [-1,1] \times \widetilde{C_{(12)}}^{st} (\R) \times [0,1] \times [-1,1]  $$
The codimension one boundary strata are given in five different ways. 
\begin{enumerate}
\item $p=(x_1,x_2) \to (-\infty,-\infty)$ in such a way that $||p||=\lambda$ remains constant. This can be achieved by $x_1=r+\sqrt 2 \lambda$ and $x_2=r,$ in which case we get $x_c(p)=r+\frac{\lambda}{\sqrt 2}$ and $||p||=\lambda.$ If we now let $r \to -\infty$ we get the desired boundary. Clearly these kind of points are scaling-invariant so we can identify these limit points with a copy of $\widehat {\mathfrak C}_2(\R)$
$$
\xymatrix{ 
\xy
(-4,-7)*{_{x_1}},
(4,-7)*{_{x_2}},
 (0,0)*{\blacktriangledown}="a",
(0,5)*{}="0",
(-4,-5)*{}="b_1",
(4,-5)*{}="b_2",
\ar @{.} "a";"0" <0pt>
\ar @{-} "a";"b_1" <0pt>
\ar @{-} "a";"b_2" <0pt>
\endxy & \ar@{->}[rr]_{\begin{array}{c} _{(x_1,x_2) \longrightarrow (-\infty,-\infty)} \\ _{||p|| \longrightarrow \lambda} \end{array}} &  &  & 
\xy
(-4,-7)*{_{x_1}},
(4,-7)*{_{x_2}},
 (0,0)*{\blacktriangleleft}="a",
(0,5)*{}="0",
(-4,-5)*{}="b_1",
(4,-5)*{}="b_2",
\ar @{.} "a";"0" <0pt>
\ar @{-} "a";"b_1" <0pt>
\ar @{-} "a";"b_2" <0pt>
\endxy
}
$$
\item Analogous to above we can consider the case when $p=(x_1,x_2) \to (\infty,\infty)$ with a fixed distance. These limit points can also be identified with a copy of $\widehat{\mathfrak C}_2(\R).$
$$
\xymatrix{ 
\xy
(-4,-7)*{_{x_1}},
(4,-7)*{_{x_2}},
 (0,0)*{\blacktriangledown}="a",
(0,5)*{}="0",
(-4,-5)*{}="b_1",
(4,-5)*{}="b_2",
\ar @{.} "a";"0" <0pt>
\ar @{-} "a";"b_1" <0pt>
\ar @{-} "a";"b_2" <0pt>
\endxy & \ar@{->}[rr]_{\begin{array}{c} _{(x_1,x_2) \longrightarrow (\infty,\infty)} \\ _{||p|| \longrightarrow \lambda} \end{array}} &  &  & 
\xy
(-4,-7)*{_{x_1}},
(4,-7)*{_{x_2}},
 (0,0)*{\blacktriangleright}="a",
(0,5)*{}="0",
(-4,-5)*{}="b_1",
(4,-5)*{}="b_2",
\ar @{.} "a";"0" <0pt>
\ar @{-} "a";"b_1" <0pt>
\ar @{-} "a";"b_2" <0pt>
\endxy
}
$$
\item $p=(x_1,x_2) \to (-\infty,a);$ limit points of this type can be identified with a copy of $\overline{C}_2(\R) \times \mathfrak{C}_1(\R) \times \overline{Conf}_1(\R).$ 
$$
\xymatrix{ 
\xy
(-4,-7)*{_{x_1}},
(4,-7)*{_{x_2}},
 (0,0)*{\blacktriangledown}="a",
(0,5)*{}="0",
(-4,-5)*{}="b_1",
(4,-5)*{}="b_2",
\ar @{.} "a";"0" <0pt>
\ar @{-} "a";"b_1" <0pt>
\ar @{-} "a";"b_2" <0pt>
\endxy & \ar@{->}[rr]_{\begin{array}{c} _{(x_1,x_2) \longrightarrow (-\infty,a)} \end{array}} &  &  & 
\xy
(-4,-12)*{_{x_1}},
(4,-12)*{_{x_2}},
 (0,0)*{\circ}="a",
(0,5)*{}="0",
(-4,-5)*{\blacktriangleleft}="b_1",
(4,-5)*{\blacktriangledown}="b_2",
(-4,-10)*{}="b_3",
(4,-10)*{}="b_4",
\ar @{.} "a";"0" <0pt>
\ar @{.} "a";"b_1" <0pt>
\ar @{.} "a";"b_2" <0pt>
\ar @{-} "b_1";"b_3" <0pt>
\ar @{-} "b_2";"b_4" <0pt>
\endxy
}
$$
\item $p=(x_1,x_2) \to (a,\infty);$  which also can be identified with a copy of $\overline{C}_2(\R) \times \overline{Conf}_1(\R) \times \mathfrak{C}_1(\R).$ 
$$
\xymatrix{ 
\xy
(-4,-7)*{_{x_1}},
(4,-7)*{_{x_2}},
 (0,0)*{\blacktriangledown}="a",
(0,5)*{}="0",
(-4,-5)*{}="b_1",
(4,-5)*{}="b_2",
\ar @{.} "a";"0" <0pt>
\ar @{-} "a";"b_1" <0pt>
\ar @{-} "a";"b_2" <0pt>
\endxy & \ar@{->}[rr]_{\begin{array}{c} _{(x_1,x_2) \longrightarrow (a,\infty)} \end{array}} &  &  & 
\xy
(-4,-12)*{_{x_1}},
(4,-12)*{_{x_2}},
 (0,0)*{\circ}="a",
(0,5)*{}="0",
(-4,-5)*{\blacktriangledown}="b_1",
(4,-5)*{\blacktriangleright}="b_2",
(-4,-10)*{}="b_3",
(4,-10)*{}="b_4",
\ar @{.} "a";"0" <0pt>
\ar @{.} "a";"b_1" <0pt>
\ar @{.} "a";"b_2" <0pt>
\ar @{-} "b_1";"b_3" <0pt>
\ar @{-} "b_2";"b_4" <0pt>
\endxy
}
$$
\item $p=(x_1,x_2) \to (a,a);$ which can be identified with a copy of $\overline{Conf}_1(\R) \times \overline{C}_2(\R).$
$$
\xymatrix{ 
\xy
(-4,-7)*{_{x_1}},
(4,-7)*{_{x_2}},
 (0,0)*{\blacktriangledown}="a",
(0,5)*{}="0",
(-4,-5)*{}="b_1",
(4,-5)*{}="b_2",
\ar @{.} "a";"0" <0pt>
\ar @{-} "a";"b_1" <0pt>
\ar @{-} "a";"b_2" <0pt>
\endxy & \ar@{->}[rr]_{\begin{array}{c} _{(x_1,x_2) \longrightarrow (a,a)} \\ _{||p|| \longrightarrow 0} \end{array}} &  &  & 
\xy
(-4,-12)*{_{x_1}},
(4,-12)*{_{x_2}},
 (0,0)*{\blacktriangledown}="a",
(0,5)*{}="0",
(0,-5)*{\bullet}="b_1",
(-4,-10)*{}="b_3",
(4,-10)*{}="b_4",
\ar @{.} "a";"0" <0pt>
\ar @{-} "a";"b_1" <0pt>
\ar @{-} "b_1";"b_3" <0pt>
\ar @{-} "b_1";"b_4" <0pt>
\endxy
}
$$
\end{enumerate}
We can summarize this in the following formula: 
\[
 \partial \xy
(-4,-7)*{_{x_1}},
(4,-7)*{_{x_2}},
 (0,0)*{\blacktriangledown}="a",
(0,5)*{}="0",
(-4,-5)*{}="b_1",
(4,-5)*{}="b_2",
\ar @{.} "a";"0" <0pt>
\ar @{-} "a";"b_1" <0pt>
\ar @{-} "a";"b_2" <0pt>
\endxy =
-\xy
(-4,-7)*{_{x_1}},
(4,-7)*{_{x_2}},
 (0,0)*{\blacktriangleleft}="a",
(0,5)*{}="0",
(-4,-5)*{}="b_1",
(4,-5)*{}="b_2",
\ar @{.} "a";"0" <0pt>
\ar @{-} "a";"b_1" <0pt>
\ar @{-} "a";"b_2" <0pt>
\endxy 
+ 
\xy
(-4,-7)*{_{x_1}},
(4,-7)*{_{x_2}},
 (0,0)*{\blacktriangleright}="a",
(0,5)*{}="0",
(-4,-5)*{}="b_1",
(4,-5)*{}="b_2",
\ar @{.} "a";"0" <0pt>
\ar @{-} "a";"b_1" <0pt>
\ar @{-} "a";"b_2" <0pt>
\endxy 
- 
\xy
(-4,-12)*{_{x_1}},
(4,-12)*{_{x_2}},
 (0,0)*{\circ}="a",
(0,5)*{}="0",
(-4,-5)*{\blacktriangleleft}="b_1",
(4,-5)*{\blacktriangledown}="b_2",
(-4,-10)*{}="b_3",
(4,-10)*{}="b_4",
\ar @{.} "a";"0" <0pt>
\ar @{.} "a";"b_1" <0pt>
\ar @{.} "a";"b_2" <0pt>
\ar @{-} "b_1";"b_3" <0pt>
\ar @{-} "b_2";"b_4" <0pt>
\endxy
 -
\xy
(-4,-12)*{_{x_1}},
(4,-12)*{_{x_2}},
 (0,0)*{\circ}="a",
(0,5)*{}="0",
(-4,-5)*{\blacktriangledown}="b_1",
(4,-5)*{\blacktriangleright}="b_2",
(-4,-10)*{}="b_3",
(4,-10)*{}="b_4",
\ar @{.} "a";"0" <0pt>
\ar @{.} "a";"b_1" <0pt>
\ar @{.} "a";"b_2" <0pt>
\ar @{-} "b_1";"b_3" <0pt>
\ar @{-} "b_2";"b_4" <0pt>
\endxy
+
\xy
(-4,-12)*{_{x_1}},
(4,-12)*{_{x_2}},
 (0,0)*{\blacktriangledown}="a",
(0,5)*{}="0",
(0,-5)*{\bullet}="b_1",
(-4,-10)*{}="b_3",
(4,-10)*{}="b_4",
\ar @{.} "a";"0" <0pt>
\ar @{-} "a";"b_1" <0pt>
\ar @{-} "b_1";"b_3" <0pt>
\ar @{-} "b_1";"b_4" <0pt>
\endxy
\]
\end{example}

\begin{example}
To further convince the reader we proceed with the example $\overline{Conf}_3(\R):$
$$ Conf_3(\R) \hookrightarrow  \displaystyle \prod_{|A| \subset [3], |A| \geq 1} \widetilde{ C}_A^{st} (\R) \times [0,1] \times [-1,1]  $$
The codimension one strata are given in twelve different ways:
\begin{enumerate}
\item $p=(x_1,x_2,x_3) \to (-\infty,-\infty,-\infty)$ in such a way that $||p_{12}|| \to \lambda_1 $ and $||p_{23}|| \to \lambda_2 $. This  can be achieved by $x_1=r-\sqrt{2} \lambda_1,$ $x_2=r,$ $x_3=r+\sqrt 2 \lambda_2$ and then letting $r \to -\infty.$ These limit points are scaling-invariant so we can identify them with a copy of $\widehat{\mathfrak C}_3(\R)$
$$
\xymatrix{ 
\xy
(-7,-7)*{_{x_1}},
(0,-7)*{_{x_2}},
(7,-7)*{_{x_3}},
 (0,0)*{\blacktriangledown}="a",
(0,5)*{}="0",
(-7,-5)*{}="b_1",
(0,-5)*{}="b_2",
(7,-5)*{}="b_3",
\ar @{.} "a";"0" <0pt>
\ar @{-} "a";"b_1" <0pt>
\ar @{-} "a";"b_2" <0pt>
\ar @{-} "a";"b_3" <0pt>
\endxy & \ar@{->}[rr]_{\begin{array}{c} _{(x_1,x_2,x_3) \longrightarrow (-\infty,-\infty,-\infty)} \\ _{||p_{12}|| \longrightarrow \lambda_1} \\ _{||p_{23}|| \longrightarrow \lambda_2 }\end{array}} &  &  & \xy
(-7,-7)*{_{x_1}},
(0,-7)*{_{x_2}},
(7,-7)*{_{x_3}},
 (0,0)*{\blacktriangleleft}="a",
(0,5)*{}="0",
(-7,-5)*{}="b_1",
(0,-5)*{}="b_2",
(7,-5)*{}="b_3",
\ar @{.} "a";"0" <0pt>
\ar @{-} "a";"b_1" <0pt>
\ar @{-} "a";"b_2" <0pt>
\ar @{-} "a";"b_3" <0pt>
\endxy
}
$$
\item Analogous to above we can consider $p=(x_1,x_2,x_3) \to (\infty,\infty,\infty)$ in such a way that $||p_{12}||=\lambda_1$ and $||p_{23}||=\lambda_2$ remains constant. This boundary strata can also be identified with a copy of $\widehat{\mathfrak C}_3(\R).$
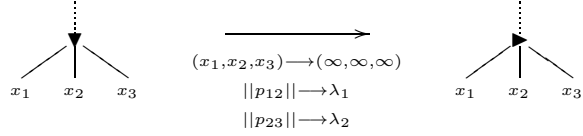
$$
\xymatrix{ 
\xy
(-7,-7)*{_{x_1}},
(0,-7)*{_{x_2}},
(7,-7)*{_{x_3}},
 (0,0)*{\blacktriangledown}="a",
(0,5)*{}="0",
(-7,-5)*{}="b_1",
(0,-5)*{}="b_2",
(7,-5)*{}="b_3",
\ar @{.} "a";"0" <0pt>
\ar @{-} "a";"b_1" <0pt>
\ar @{-} "a";"b_2" <0pt>
\ar @{-} "a";"b_3" <0pt>
\endxy & \ar@{->}[rr]_{\begin{array}{c} _{(x_1,x_2,x_3) \longrightarrow (\infty,\infty,\infty)} \\ _{||p_{12}|| \longrightarrow \lambda_1} \\ _{||p_{23}|| \longrightarrow \lambda_2}\end{array}} &  &  & \xy
(-7,-7)*{_{x_1}},
(0,-7)*{_{x_2}},
(7,-7)*{_{x_3}},
 (0,0)*{\blacktriangleright}="a",
(0,5)*{}="0",
(-7,-5)*{}="b_1",
(0,-5)*{}="b_2",
(7,-5)*{}="b_3",
\ar @{.} "a";"0" <0pt>
\ar @{-} "a";"b_1" <0pt>
\ar @{-} "a";"b_2" <0pt>
\ar @{-} "a";"b_3" <0pt>
\endxy
}
$$
\item  $p=(x_1,x_2,x_3) \to (-\infty,-\infty,a)$ in such a way that  $||p_{12}||$ diverges. Points of this type can be identified with $\overline{C}_3(\R) \times \widehat{\mathfrak C}_1(\R) \times \widehat{\mathfrak C}_1(\R) \times \overline{Conf}_1(\R).$
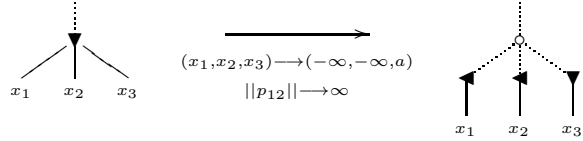
$$
\xymatrix{ 
\xy
(-7,-7)*{_{x_1}},
(0,-7)*{_{x_2}},
(7,-7)*{_{x_3}},
 (0,0)*{\blacktriangledown}="a",
(0,5)*{}="0",
(-7,-5)*{}="b_1",
(0,-5)*{}="b_2",
(7,-5)*{}="b_3",
\ar @{.} "a";"0" <0pt>
\ar @{-} "a";"b_1" <0pt>
\ar @{-} "a";"b_2" <0pt>
\ar @{-} "a";"b_3" <0pt>
\endxy & \ar@{->}[rr]_{\begin{array}{c} _{(x_1,x_2,x_3) \longrightarrow (-\infty,-\infty,a)} \\ _{||p_{12}|| \longrightarrow \infty} \\ \end{array}} &  &  & \xy
(-7,-12)*{_{x_1}},
(0,-12)*{_{x_2}},
(7,-12)*{_{x_3}},
 (0,0)*{\circ}="a",
(0,5)*{}="0",
(-7,-5)*{\blacktriangleleft}="b_1",
(0,-5)*{\blacktriangleleft}="b_2",
(7,-5)*{\blacktriangledown}="b_3",
(-7,-5)*{}="b_1",
(0,-5)*{}="b_2",
(7,-5)*{}="b_3",
(-7,-10)*{}="b_4",
(0,-10)*{}="b_5",
(7,-10)*{}="b_6",
\ar @{.} "a";"0" <0pt>
\ar @{.} "a";"b_1" <0pt>
\ar @{.} "a";"b_2" <0pt>
\ar @{.} "a";"b_3" <0pt>
\ar @{-} "b_1";"b_4" <0pt>
\ar @{-} "b_2";"b_5" <0pt>
\ar @{-} "b_3";"b_6" <0pt>
\endxy
}
$$
\item  $p=(x_1,x_2,x_3) \to (a,\infty,\infty)$ in such a way that  $||p_{23}||$ diverges. Points of this type can be identified with $\overline{C}_3(\R) \times \overline{Conf}_1(\R) \times \widehat{\mathfrak C}_1(\R) \times \widehat{\mathfrak C}_1(\R).$
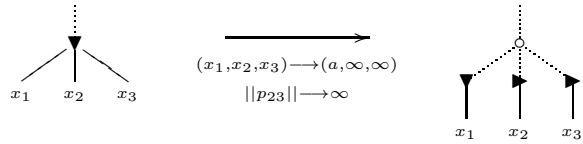
$$
\xymatrix{ 
\xy
(-7,-7)*{_{x_1}},
(0,-7)*{_{x_2}},
(7,-7)*{_{x_3}},
 (0,0)*{\blacktriangledown}="a",
(0,5)*{}="0",
(-7,-5)*{}="b_1",
(0,-5)*{}="b_2",
(7,-5)*{}="b_3",
\ar @{.} "a";"0" <0pt>
\ar @{-} "a";"b_1" <0pt>
\ar @{-} "a";"b_2" <0pt>
\ar @{-} "a";"b_3" <0pt>
\endxy & \ar@{->}[rr]_{\begin{array}{c} _{(x_1,x_2,x_3) \longrightarrow (a,\infty,\infty)} \\ _{||p_{23}|| \longrightarrow \infty} \\ \end{array}} &  &  & \xy
(-7,-12)*{_{x_1}},
(0,-12)*{_{x_2}},
(7,-12)*{_{x_3}},
 (0,0)*{\circ}="a",
(0,5)*{}="0",
(-7,-5)*{\blacktriangledown}="b_1",
(0,-5)*{\blacktriangleright}="b_2",
(7,-5)*{\blacktriangleright}="b_3",
(-7,-5)*{}="b_1",
(0,-5)*{}="b_2",
(7,-5)*{}="b_3",
(-7,-10)*{}="b_4",
(0,-10)*{}="b_5",
(7,-10)*{}="b_6",
\ar @{.} "a";"0" <0pt>
\ar @{.} "a";"b_1" <0pt>
\ar @{.} "a";"b_2" <0pt>
\ar @{.} "a";"b_3" <0pt>
\ar @{-} "b_1";"b_4" <0pt>
\ar @{-} "b_2";"b_5" <0pt>
\ar @{-} "b_3";"b_6" <0pt>
\endxy
}
$$
\item  $p=(x_1,x_2,x_3) \to (-\infty,-\infty,a)$ in such a way that  $||p_{12}||$ converges. Points of this type can be identified with $\overline{C}_2(\R) \times \widehat{\mathfrak C}_2(\R) \times \overline{Conf}_1(\R).$
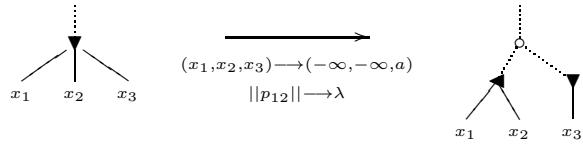
$$
\xymatrix{ 
\xy
(-7,-7)*{_{x_1}},
(0,-7)*{_{x_2}},
(7,-7)*{_{x_3}},
 (0,0)*{\blacktriangledown}="a",
(0,5)*{}="0",
(-7,-5)*{}="b_1",
(0,-5)*{}="b_2",
(7,-5)*{}="b_3",
\ar @{.} "a";"0" <0pt>
\ar @{-} "a";"b_1" <0pt>
\ar @{-} "a";"b_2" <0pt>
\ar @{-} "a";"b_3" <0pt>
\endxy & \ar@{->}[rr]_{\begin{array}{c} _{(x_1,x_2,x_3) \longrightarrow (-\infty,-\infty,a)} \\ _{||p_{12}|| \longrightarrow \lambda} \end{array}} &  &  & \xy
(-7,-12)*{_{x_1}},
(0,-12)*{_{x_2}},
(7,-12)*{_{x_3}},
 (0,0)*{\circ}="a",
(0,5)*{}="0",
(-3,-5)*{\blacktriangleleft}="b_2",
(7,-5)*{\blacktriangledown}="b_3",
(-3,-5)*{}="b_2",
(7,-5)*{}="b_3",
(-7,-10)*{}="b_4",
(0,-10)*{}="b_5",
(7,-10)*{}="b_6",
\ar @{.} "a";"0" <0pt>
\ar @{.} "a";"b_2" <0pt>
\ar @{.} "a";"b_3" <0pt>
\ar @{-} "b_2";"b_4" <0pt>
\ar @{-} "b_2";"b_5" <0pt>
\ar @{-} "b_3";"b_6" <0pt>
\endxy
}
$$
\item $p=(x_1,x_2,x_3) \to (-\infty,a,b)$ where $a<b;$ which can be identified with $\overline{C}_3(\R) \times \widehat{\mathfrak C}_1(\R) \times \overline{Conf}_2(\R)$
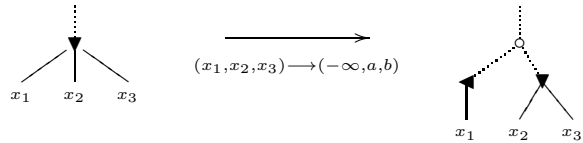
$$
\xymatrix{ 
\xy
(-7,-7)*{_{x_1}},
(0,-7)*{_{x_2}},
(7,-7)*{_{x_3}},
 (0,0)*{\blacktriangledown}="a",
(0,5)*{}="0",
(-7,-5)*{}="b_1",
(0,-5)*{}="b_2",
(7,-5)*{}="b_3",
\ar @{.} "a";"0" <0pt>
\ar @{-} "a";"b_1" <0pt>
\ar @{-} "a";"b_2" <0pt>
\ar @{-} "a";"b_3" <0pt>
\endxy & \ar@{->}[rr]_{\begin{array}{c} _{(x_1,x_2,x_3) \longrightarrow (-\infty,a,b)} \end{array}} &  &  & 
\xy
(-7,-12)*{_{x_1}},
(0,-12)*{_{x_2}},
(7,-12)*{_{x_3}},
 (0,0)*{\circ}="a",
(0,5)*{}="0",
(-7,-5)*{\blacktriangleleft}="b_1",
(3,-5)*{\blacktriangledown}="b_2",
(-7,-5)*{}="b_1",
(3,-5)*{}="b_2",
(-7,-10)*{}="b_4",
(0,-10)*{}="b_5",
(7,-10)*{}="b_6",
\ar @{.} "a";"0" <0pt>
\ar @{.} "a";"b_1" <0pt>
\ar @{.} "a";"b_2" <0pt>
\ar @{-} "b_1";"b_4" <0pt>
\ar @{-} "b_2";"b_5" <0pt>
\ar @{-} "b_2";"b_6" <0pt>
\endxy
}
$$
\item $p=(x_1,x_2,x_3) \to (-\infty,a,\infty).$ These points can be identified with $\overline{C}_3(\R) \times \widehat{\mathfrak C}_1(\R) \times \overline{Conf}_1(\R)  \times \widehat{\mathfrak C}_1(\R)$  
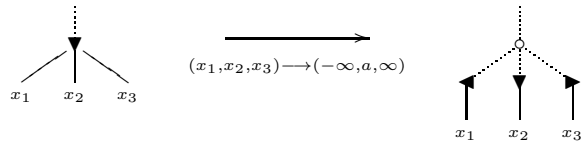
$$
\xymatrix{ 
\xy
(-7,-7)*{_{x_1}},
(0,-7)*{_{x_2}},
(7,-7)*{_{x_3}},
 (0,0)*{\blacktriangledown}="a",
(0,5)*{}="0",
(-7,-5)*{}="b_1",
(0,-5)*{}="b_2",
(7,-5)*{}="b_3",
\ar @{.} "a";"0" <0pt>
\ar @{-} "a";"b_1" <0pt>
\ar @{-} "a";"b_2" <0pt>
\ar @{-} "a";"b_3" <0pt>
\endxy & \ar@{->}[rr]_{\begin{array}{c} _{(x_1,x_2,x_3) \longrightarrow (-\infty,a,\infty)} \end{array}} &  &  & 
\xy
(-7,-12)*{_{x_1}},
(0,-12)*{_{x_2}},
(7,-12)*{_{x_3}},
 (0,0)*{\circ}="a",
(0,5)*{}="0",
(-7,-5)*{\blacktriangleleft}="b_1",
(0,-5)*{\blacktriangledown}="b_2",
(7,-5)*{\blacktriangleright}="b_3",
(-7,-5)*{}="b_1",
(0,-5)*{}="b_2",
(7,-5)*{}="b_3",
(-7,-10)*{}="b_4",
(0,-10)*{}="b_5",
(7,-10)*{}="b_6",
\ar @{.} "a";"0" <0pt>
\ar @{.} "a";"b_1" <0pt>
\ar @{.} "a";"b_2" <0pt>
\ar @{.} "a";"b_3" <0pt>
\ar @{-} "b_1";"b_4" <0pt>
\ar @{-} "b_2";"b_5" <0pt>
\ar @{-} "b_3";"b_6" <0pt>
\endxy
}
$$
\item $p=(x_1,x_2,x_3) \to (a,b,\infty),$ where $a<b.$ Points of this type can be identified with $\overline{C}_2(\R)  \times \overline{Conf}_2(\R) \times \widehat{\mathfrak C}_1(\R).$
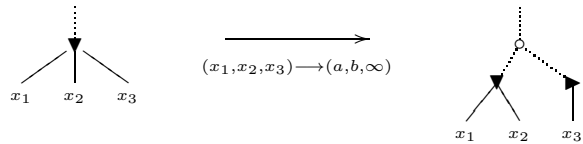
$$
\xymatrix{ 
\xy
(-7,-7)*{_{x_1}},
(0,-7)*{_{x_2}},
(7,-7)*{_{x_3}},
 (0,0)*{\blacktriangledown}="a",
(0,5)*{}="0",
(-7,-5)*{}="b_1",
(0,-5)*{}="b_2",
(7,-5)*{}="b_3",
\ar @{.} "a";"0" <0pt>
\ar @{-} "a";"b_1" <0pt>
\ar @{-} "a";"b_2" <0pt>
\ar @{-} "a";"b_3" <0pt>
\endxy & \ar@{->}[rr]_{\begin{array}{c} _{(x_1,x_2,x_3) \longrightarrow (a,b,\infty)} \end{array}} &  &  & 
\xy
(-7,-12)*{_{x_1}},
(0,-12)*{_{x_2}},
(7,-12)*{_{x_3}},
 (0,0)*{\circ}="a",
(0,5)*{}="0",
(-3,-5)*{\blacktriangledown}="b_2",
(7,-5)*{\blacktriangleright}="b_3",
(-3,-5)*{}="b_2",
(7,-5)*{}="b_3",
(-7,-10)*{}="b_4",
(0,-10)*{}="b_5",
(7,-10)*{}="b_6",
\ar @{.} "a";"0" <0pt>
\ar @{.} "a";"b_2" <0pt>
\ar @{.} "a";"b_3" <0pt>
\ar @{-} "b_2";"b_4" <0pt>
\ar @{-} "b_2";"b_5" <0pt>
\ar @{-} "b_3";"b_6" <0pt>
\endxy
}
$$
\item $p=(x_1,x_2,x_3) \to (a,\infty,\infty)$ in such a way that $||p_{23}||$ converges. This boundary strata can be identified with $\overline{C}_2(\R)  \times \overline{Conf}_2(\R) \times \widehat{\mathfrak C}_1(\R)$
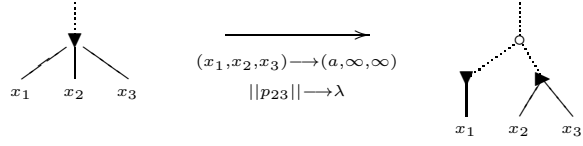
$$
\xymatrix{ 
\xy
(-7,-7)*{_{x_1}},
(0,-7)*{_{x_2}},
(7,-7)*{_{x_3}},
 (0,0)*{\blacktriangledown}="a",
(0,5)*{}="0",
(-7,-5)*{}="b_1",
(0,-5)*{}="b_2",
(7,-5)*{}="b_3",
\ar @{.} "a";"0" <0pt>
\ar @{-} "a";"b_1" <0pt>
\ar @{-} "a";"b_2" <0pt>
\ar @{-} "a";"b_3" <0pt>
\endxy & \ar@{->}[rr]_{\begin{array}{c} _{(x_1,x_2,x_3) \longrightarrow (a,\infty,\infty)} \\ _{||p_{23}|| \longrightarrow \lambda} \end{array}} &  &  & \xy
(-7,-12)*{_{x_1}},
(0,-12)*{_{x_2}},
(7,-12)*{_{x_3}},
 (0,0)*{\circ}="a",
(0,5)*{}="0",
(-7,-5)*{\blacktriangledown}="b_1",
(3,-5)*{\blacktriangleright}="b_2",
(-7,-5)*{}="b_1",
(3,-5)*{}="b_2",
(-7,-10)*{}="b_4",
(0,-10)*{}="b_5",
(7,-10)*{}="b_6",
\ar @{.} "a";"0" <0pt>
\ar @{.} "a";"b_1" <0pt>
\ar @{.} "a";"b_2" <0pt>
\ar @{-} "b_1";"b_4" <0pt>
\ar @{-} "b_2";"b_5" <0pt>
\ar @{-} "b_2";"b_6" <0pt>
\endxy
}
$$
\item $p=(x_1,x_2,x_3) \to (a,a,b),$ where $a<b.$ Points of this type can be identified with $ \overline{Conf}_2(\R) \times \overline{C}_2(\R).$
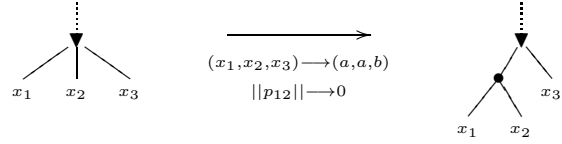
$$
\xymatrix{ 
\xy
(-7,-7)*{_{x_1}},
(0,-7)*{_{x_2}},
(7,-7)*{_{x_3}},
 (0,0)*{\blacktriangledown}="a",
(0,5)*{}="0",
(-7,-5)*{}="b_1",
(0,-5)*{}="b_2",
(7,-5)*{}="b_3",
\ar @{.} "a";"0" <0pt>
\ar @{-} "a";"b_1" <0pt>
\ar @{-} "a";"b_2" <0pt>
\ar @{-} "a";"b_3" <0pt>
\endxy & \ar@{->}[rr]_{\begin{array}{c} _{(x_1,x_2,x_3) \longrightarrow (a,a,b)} \\ _{||p_{12}|| \longrightarrow 0}\end{array}} &  &  & \xy
(-7,-12)*{_{x_1}},
(0,-12)*{_{x_2}},
(4,-7)*{_{x_3}},
 (0,0)*{\blacktriangledown}="a",
(0,5)*{}="0",
(-3,-5)*{\bullet}="b_2",
(-3,-5)*{}="b_2",
(7,-5)*{}="b_3",
(-7,-10)*{}="b_4",
(0,-10)*{}="b_5",
(4,-5)*{}="b_6",
\ar @{.} "a";"0" <0pt>
\ar @{-} "a";"b_2" <0pt>
\ar @{-} "a";"b_6" <0pt>
\ar @{-} "b_2";"b_4" <0pt>
\ar @{-} "b_2";"b_5" <0pt>
\endxy
}
$$
\item $p=(x_1,x_2,x_3) \to (a,b,b),$ where $a<b.$ This boundary strata can be identified with $ \overline{Conf}_2(\R) \times \overline{C}_2(\R)$
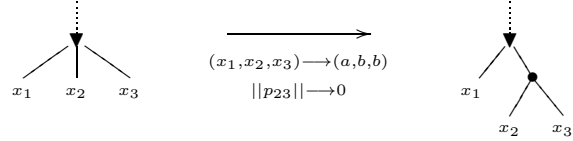
$$
\xymatrix{ 
\xy
(-7,-7)*{_{x_1}},
(0,-7)*{_{x_2}},
(7,-7)*{_{x_3}},
 (0,0)*{\blacktriangledown}="a",
(0,5)*{}="0",
(-7,-5)*{}="b_1",
(0,-5)*{}="b_2",
(7,-5)*{}="b_3",
\ar @{.} "a";"0" <0pt>
\ar @{-} "a";"b_1" <0pt>
\ar @{-} "a";"b_2" <0pt>
\ar @{-} "a";"b_3" <0pt>
\endxy & \ar@{->}[rr]_{\begin{array}{c} _{(x_1,x_2,x_3) \longrightarrow (a,b,b)} \\ _{||p_{23}|| \longrightarrow 0} \end{array}} &  &  & \xy
(-5,-7)*{_{x_1}},
(0,-12)*{_{x_2}},
(7,-12)*{_{x_3}},
 (0,0)*{\blacktriangledown}="a",
(0,5)*{}="0",
(3,-5)*{\bullet}="b_2",
(-7,-5)*{}="b_1",
(3,-5)*{}="b_2",
(-4,-5)*{}="b_4",
(0,-10)*{}="b_5",
(7,-10)*{}="b_6",
\ar @{.} "a";"0" <0pt>
\ar @{-} "a";"b_4" <0pt>
\ar @{-} "a";"b_2" <0pt>
\ar @{-} "b_2";"b_5" <0pt>
\ar @{-} "b_2";"b_6" <0pt>
\endxy
}
$$
\item $p=(x_1,x_2,x_3) \to (a,a,a).$ This boundary strata can be identified with $ \overline{Conf}_1(\R) \times \overline{C}_3(\R)$
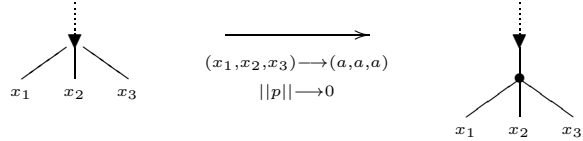
$$
\xymatrix{ 
\xy
(-7,-7)*{_{x_1}},
(0,-7)*{_{x_2}},
(7,-7)*{_{x_3}},
 (0,0)*{\blacktriangledown}="a",
(0,5)*{}="0",
(-7,-5)*{}="b_1",
(0,-5)*{}="b_2",
(7,-5)*{}="b_3",
\ar @{.} "a";"0" <0pt>
\ar @{-} "a";"b_1" <0pt>
\ar @{-} "a";"b_2" <0pt>
\ar @{-} "a";"b_3" <0pt>
\endxy & \ar@{->}[rr]_{\begin{array}{c} _{(x_1,x_2,x_3) \longrightarrow (a,a,a)} \\ _{||p|| \longrightarrow 0} \end{array}} &  &  & \xy
(-7,-12)*{_{x_1}},
(0,-12)*{_{x_2}},
(7,-12)*{_{x_3}},
 (0,0)*{\blacktriangledown}="a",
(0,5)*{}="0",
(0,-5)*{\bullet}="b_2",
(0,-5)*{}="b_2",
(-7,-10)*{}="b_4",
(0,-10)*{}="b_5",
(7,-10)*{}="b_6",
\ar @{.} "a";"0" <0pt>
\ar @{-} "a";"b_2" <0pt>
\ar @{-} "b_2";"b_4" <0pt>
\ar @{-} "b_2";"b_5" <0pt>
\ar @{-} "b_2";"b_6" <0pt>
\endxy
}
$$
\end{enumerate}
We summarize this in the formula: \[
\begin{split}
\partial \xy
(-7,-7)*{_{x_1}},
(0,-7)*{_{x_2}},
(7,-7)*{_{x_3}},
 (0,0)*{\blacktriangledown}="a",
(0,5)*{}="0",
(-7,-5)*{}="b_1",
(0,-5)*{}="b_2",
(7,-5)*{}="b_3",
\ar @{.} "a";"0" <0pt>
\ar @{-} "a";"b_1" <0pt>
\ar @{-} "a";"b_2" <0pt>
\ar @{-} "a";"b_3" <0pt>
\endxy
 & =
-
\xy
(-7,-7)*{_{x_1}},
(0,-7)*{_{x_2}},
(7,-7)*{_{x_3}},
 (0,0)*{\blacktriangleleft}="a",
(0,5)*{}="0",
(-7,-5)*{}="b_1",
(0,-5)*{}="b_2",
(7,-5)*{}="b_3",
\ar @{.} "a";"0" <0pt>
\ar @{-} "a";"b_1" <0pt>
\ar @{-} "a";"b_2" <0pt>
\ar @{-} "a";"b_3" <0pt>
\endxy
+
\xy
(-7,-7)*{_{x_1}},
(0,-7)*{_{x_2}},
(7,-7)*{_{x_3}},
 (0,0)*{\blacktriangleright}="a",
(0,5)*{}="0",
(-7,-5)*{}="b_1",
(0,-5)*{}="b_2",
(7,-5)*{}="b_3",
\ar @{.} "a";"0" <0pt>
\ar @{-} "a";"b_1" <0pt>
\ar @{-} "a";"b_2" <0pt>
\ar @{-} "a";"b_3" <0pt>
\endxy
 + \xy
(-7,-12)*{_{x_1}},
(0,-12)*{_{x_2}},
(7,-12)*{_{x_3}},
 (0,0)*{\circ}="a",
(0,5)*{}="0",
(-7,-5)*{\blacktriangleleft}="b_1",
(0,-5)*{\blacktriangleleft}="b_2",
(7,-5)*{\blacktriangledown}="b_3",
(-7,-5)*{}="b_1",
(0,-5)*{}="b_2",
(7,-5)*{}="b_3",
(-7,-10)*{}="b_4",
(0,-10)*{}="b_5",
(7,-10)*{}="b_6",
\ar @{.} "a";"0" <0pt>
\ar @{.} "a";"b_1" <0pt>
\ar @{.} "a";"b_2" <0pt>
\ar @{.} "a";"b_3" <0pt>
\ar @{-} "b_1";"b_4" <0pt>
\ar @{-} "b_2";"b_5" <0pt>
\ar @{-} "b_3";"b_6" <0pt>
\endxy
+
\xy
(-7,-12)*{_{x_1}},
(0,-12)*{_{x_2}},
(7,-12)*{_{x_3}},
 (0,0)*{\circ}="a",
(0,5)*{}="0",
(-7,-5)*{\blacktriangledown}="b_1",
(0,-5)*{\blacktriangleright}="b_2",
(7,-5)*{\blacktriangleright}="b_3",
(-7,-5)*{}="b_1",
(0,-5)*{}="b_2",
(7,-5)*{}="b_3",
(-7,-10)*{}="b_4",
(0,-10)*{}="b_5",
(7,-10)*{}="b_6",
\ar @{.} "a";"0" <0pt>
\ar @{.} "a";"b_1" <0pt>
\ar @{.} "a";"b_2" <0pt>
\ar @{.} "a";"b_3" <0pt>
\ar @{-} "b_1";"b_4" <0pt>
\ar @{-} "b_2";"b_5" <0pt>
\ar @{-} "b_3";"b_6" <0pt>
\endxy \\ &
+
\xy
(-7,-12)*{_{x_1}},
(0,-12)*{_{x_2}},
(7,-12)*{_{x_3}},
 (0,0)*{\circ}="a",
(0,5)*{}="0",
(-3,-5)*{\blacktriangleleft}="b_2",
(7,-5)*{\blacktriangledown}="b_3",
(-3,-5)*{}="b_2",
(7,-5)*{}="b_3",
(-7,-10)*{}="b_4",
(0,-10)*{}="b_5",
(7,-10)*{}="b_6",
\ar @{.} "a";"0" <0pt>
\ar @{.} "a";"b_2" <0pt>
\ar @{.} "a";"b_3" <0pt>
\ar @{-} "b_2";"b_4" <0pt>
\ar @{-} "b_2";"b_5" <0pt>
\ar @{-} "b_3";"b_6" <0pt>
\endxy
+
\xy
(-7,-12)*{_{x_1}},
(0,-12)*{_{x_2}},
(7,-12)*{_{x_3}},
 (0,0)*{\circ}="a",
(0,5)*{}="0",
(-7,-5)*{\blacktriangleleft}="b_1",
(3,-5)*{\blacktriangledown}="b_2",
(-7,-5)*{}="b_1",
(3,-5)*{}="b_2",
(-7,-10)*{}="b_4",
(0,-10)*{}="b_5",
(7,-10)*{}="b_6",
\ar @{.} "a";"0" <0pt>
\ar @{.} "a";"b_1" <0pt>
\ar @{.} "a";"b_2" <0pt>
\ar @{-} "b_1";"b_4" <0pt>
\ar @{-} "b_2";"b_5" <0pt>
\ar @{-} "b_2";"b_6" <0pt>
\endxy
+
\xy
(-7,-12)*{_{x_1}},
(0,-12)*{_{x_2}},
(7,-12)*{_{x_3}},
 (0,0)*{\circ}="a",
(0,5)*{}="0",
(-7,-5)*{\blacktriangleleft}="b_1",
(0,-5)*{\blacktriangledown}="b_2",
(7,-5)*{\blacktriangleright}="b_3",
(-7,-5)*{}="b_1",
(0,-5)*{}="b_2",
(7,-5)*{}="b_3",
(-7,-10)*{}="b_4",
(0,-10)*{}="b_5",
(7,-10)*{}="b_6",
\ar @{.} "a";"0" <0pt>
\ar @{.} "a";"b_1" <0pt>
\ar @{.} "a";"b_2" <0pt>
\ar @{.} "a";"b_3" <0pt>
\ar @{-} "b_1";"b_4" <0pt>
\ar @{-} "b_2";"b_5" <0pt>
\ar @{-} "b_3";"b_6" <0pt>
\endxy
-\xy
(-7,-12)*{_{x_1}},
(0,-12)*{_{x_2}},
(7,-12)*{_{x_3}},
 (0,0)*{\circ}="a",
(0,5)*{}="0",
(-3,-5)*{\blacktriangledown}="b_2",
(7,-5)*{\blacktriangleright}="b_3",
(-3,-5)*{}="b_2",
(7,-5)*{}="b_3",
(-7,-10)*{}="b_4",
(0,-10)*{}="b_5",
(7,-10)*{}="b_6",
\ar @{.} "a";"0" <0pt>
\ar @{.} "a";"b_2" <0pt>
\ar @{.} "a";"b_3" <0pt>
\ar @{-} "b_2";"b_4" <0pt>
\ar @{-} "b_2";"b_5" <0pt>
\ar @{-} "b_3";"b_6" <0pt>
\endxy
\\ & + \xy
(-7,-12)*{_{x_1}},
(0,-12)*{_{x_2}},
(7,-12)*{_{x_3}},
 (0,0)*{\circ}="a",
(0,5)*{}="0",
(-7,-5)*{\blacktriangledown}="b_1",
(3,-5)*{\blacktriangleright}="b_2",
(-7,-5)*{}="b_1",
(3,-5)*{}="b_2",
(-7,-10)*{}="b_4",
(0,-10)*{}="b_5",
(7,-10)*{}="b_6",
\ar @{.} "a";"0" <0pt>
\ar @{.} "a";"b_1" <0pt>
\ar @{.} "a";"b_2" <0pt>
\ar @{-} "b_1";"b_4" <0pt>
\ar @{-} "b_2";"b_5" <0pt>
\ar @{-} "b_2";"b_6" <0pt>
\endxy
-\xy
(-7,-12)*{_{x_1}},
(0,-12)*{_{x_2}},
(4,-7)*{_{x_3}},
 (0,0)*{\blacktriangledown}="a",
(0,5)*{}="0",
(-3,-5)*{\bullet}="b_2",
(-3,-5)*{}="b_2",
(7,-5)*{}="b_3",
(-7,-10)*{}="b_4",
(0,-10)*{}="b_5",
(4,-5)*{}="b_6",
\ar @{.} "a";"0" <0pt>
\ar @{-} "a";"b_2" <0pt>
\ar @{-} "a";"b_6" <0pt>
\ar @{-} "b_2";"b_4" <0pt>
\ar @{-} "b_2";"b_5" <0pt>
\endxy
+
\xy
(-5,-7)*{_{x_1}},
(0,-12)*{_{x_2}},
(7,-12)*{_{x_3}},
 (0,0)*{\blacktriangledown}="a",
(0,5)*{}="0",
(3,-5)*{\bullet}="b_2",
(-7,-5)*{}="b_1",
(3,-5)*{}="b_2",
(-4,-5)*{}="b_4",
(0,-10)*{}="b_5",
(7,-10)*{}="b_6",
\ar @{.} "a";"0" <0pt>
\ar @{-} "a";"b_4" <0pt>
\ar @{-} "a";"b_2" <0pt>
\ar @{-} "b_2";"b_5" <0pt>
\ar @{-} "b_2";"b_6" <0pt>
\endxy
+\xy
(-7,-12)*{_{x_1}},
(0,-12)*{_{x_2}},
(7,-12)*{_{x_3}},
 (0,0)*{\blacktriangledown}="a",
(0,5)*{}="0",
(0,-5)*{\bullet}="b_2",
(0,-5)*{}="b_2",
(-7,-10)*{}="b_4",
(0,-10)*{}="b_5",
(7,-10)*{}="b_6",
\ar @{.} "a";"0" <0pt>
\ar @{-} "a";"b_2" <0pt>
\ar @{-} "b_2";"b_4" <0pt>
\ar @{-} "b_2";"b_5" <0pt>
\ar @{-} "b_2";"b_6" <0pt>
\endxy
\end{split}
\]
\end{example}
We summarize the result in our main theorem
\begin{theorem} \label{hoass}
The face complex on the disjoint union $\overline{C}_\bullet(\R) \sqcup \widehat{\mathfrak{C}}_{\bullet}(\R) \sqcup \overline{Conf}_\bullet (\R) \sqcup \widehat{\mathfrak{C}}_{\bullet}(\R) \sqcup \overline{C}_\bullet(\R)$ is naturally a dg free operad of transformation type 
\[ \begin{split}
\mathcal Ho(As)_\infty := 
\mathcal Free  \left \langle 
\vphantom{\xy
(1,-5)*{\ldots},
(-13,-7)*{_{i_1}},
(-8,-7)*{_{i_2}},
(-3,-7)*{_{i_3}},
(7,-7)*{_{i_{q-1}}},
(13,-7)*{_{i_q}},
 (0,0)*{\bullet}="a",
(0,5)*{}="0",
(-12,-5)*{}="b_1",
(-8,-5)*{}="b_2",
(-3,-5)*{}="b_3",
(8,-5)*{}="b_4",
(12,-5)*{}="b_5",
\ar @{-} "a";"0" <0pt>
\ar @{-} "a";"b_2" <0pt>
\ar @{-} "a";"b_3" <0pt>
\ar @{-} "a";"b_1" <0pt>
\ar @{-} "a";"b_4" <0pt>
\ar @{-} "a";"b_5" <0pt>
\endxy} \right.  & 
\xy
(1,-5)*{\ldots},
(-13,-7)*{_{i_1}},
(-8,-7)*{_{i_2}},
(-3,-7)*{_{i_3}},
(7,-7)*{_{i_{q-1}}},
(13,-7)*{_{i_q}},
 (0,0)*{\bullet}="a",
(0,5)*{}="0",
(-12,-5)*{}="b_1",
(-8,-5)*{}="b_2",
(-3,-5)*{}="b_3",
(8,-5)*{}="b_4",
(12,-5)*{}="b_5",
\ar @{-} "a";"0" <0pt>
\ar @{-} "a";"b_2" <0pt>
\ar @{-} "a";"b_3" <0pt>
\ar @{-} "a";"b_1" <0pt>
\ar @{-} "a";"b_4" <0pt>
\ar @{-} "a";"b_5" <0pt>
\endxy ,
\xy
(1,-5)*{\ldots},
(-13,-7)*{_{i_1}},
(-8,-7)*{_{i_2}},
(-3,-7)*{_{i_3}},
(7,-7)*{_{i_{k-1}}},
(13,-7)*{_{i_k}},
 (0,0)*{\blacktriangleleft}="a",
(0,5)*{}="0",
(-12,-5)*{}="b_1",
(-8,-5)*{}="b_2",
(-3,-5)*{}="b_3",  
(8,-5)*{}="b_4",
(12,-5)*{}="b_5",
\ar @{.} "a";"0" <0pt>
\ar @{-} "a";"b_2" <0pt>
\ar @{-} "a";"b_3" <0pt>
\ar @{-} "a";"b_1" <0pt>
\ar @{-} "a";"b_4" <0pt>
\ar @{-} "a";"b_5" <0pt>
\endxy
,
\xy
(1,-5)*{\ldots},
(-13,-7)*{_{i_1}},
(-8,-7)*{_{i_2}},
(-3,-7)*{_{i_3}},
(7,-7)*{_{i_{m-1}}},
(13,-7)*{_{i_m}},
 (0,0)*{\blacktriangledown}="a",
(0,5)*{}="0",
(-12,-5)*{}="b_1",
(-8,-5)*{}="b_2",
(-3,-5)*{}="b_3",
(8,-5)*{}="b_4",
(12,-5)*{}="b_5",
\ar @{.} "a";"0" <0pt>
\ar @{-} "a";"b_2" <0pt>
\ar @{-} "a";"b_3" <0pt>
\ar @{-} "a";"b_1" <0pt>
\ar @{-} "a";"b_4" <0pt>
\ar @{-} "a";"b_5" <0pt>
\endxy
,\\ & 
\xy
(1,-5)*{\ldots},
(-13,-7)*{_{i_1}},
(-8,-7)*{_{i_2}},
(-3,-7)*{_{i_3}},
(7,-7)*{_{i_{n-1}}},
(13,-7)*{_{i_n}},
 (0,0)*{\blacktriangleright}="a",
(0,5)*{}="0",
(-12,-5)*{}="b_1",
(-8,-5)*{}="b_2",
(-3,-5)*{}="b_3",
(8,-5)*{}="b_4",
(12,-5)*{}="b_5",
\ar @{.} "a";"0" <0pt>
\ar @{-} "a";"b_2" <0pt>
\ar @{-} "a";"b_3" <0pt>
\ar @{-} "a";"b_1" <0pt>
\ar @{-} "a";"b_4" <0pt>
\ar @{-} "a";"b_5" <0pt>
\endxy,
\xy
(1,-5)*{\ldots},
(-13,-7)*{_{i_1}},
(-8,-7)*{_{i_2}},
(-3,-7)*{_{i_3}},
(7,-7)*{_{i_{p-1}}},
(13,-7)*{_{i_p}},
 (0,0)*{\circ}="a",
(0,5)*{}="0",
(-12,-5)*{}="b_1",
(-8,-5)*{}="b_2",
(-3,-5)*{}="b_3",
(8,-5)*{}="b_4",
(12,-5)*{}="b_5",
\ar @{.} "a";"0" <0pt>
\ar @{.} "a";"b_2" <0pt>
\ar @{.} "a";"b_3" <0pt>
\ar @{.} "a";"b_1" <0pt>
\ar @{.} "a";"b_4" <0pt>
\ar @{.} "a";"b_5" <0pt>
\endxy 
\left. \vphantom{\xy
(1,-5)*{\ldots},
(-13,-7)*{_{i_1}},
(-8,-7)*{_{i_2}},
(-3,-7)*{_{i_3}},
(7,-7)*{_{i_{q-1}}},
(13,-7)*{_{i_q}},
 (0,0)*{\bullet}="a",
(0,5)*{}="0",
(-12,-5)*{}="b_1",
(-8,-5)*{}="b_2",
(-3,-5)*{}="b_3",
(8,-5)*{}="b_4",
(12,-5)*{}="b_5",
\ar @{-} "a";"0" <0pt>
\ar @{-} "a";"b_2" <0pt>
\ar @{-} "a";"b_3" <0pt>
\ar @{-} "a";"b_1" <0pt>
\ar @{-} "a";"b_4" <0pt>
\ar @{-} "a";"b_5" <0pt>
\endxy} \right \rangle_{p,q\geq 2,k,m,n \geq 1}.
 \end{split}
\]
Representation of this operad in a pair of vector spaces $V^1$ and $V^2$ is the structure of two $A_\infty$ algebras, $(V^1,\mu^1)$ and $(V^2,\mu^2)$, two $A_\infty$ morphisms, $f,g:(V^1,\mu^1)\to (V^2,\mu^2)$ and a homotopy $h$ between the morphism $h:f \to g.$ The action of the differential was described earlier.
\end{theorem}
\begin{proof}
The proof is by inspection. We have worked out the cases of two and three points in detail and we can see that they correspond the to algebraic formulas of the previous section. The general case is treated in complete analogy. 

Let $p$ be a configuration of $n$ points on the real line $p=(x_1 <x_2 < \ldots <x_n) \in Conf_n(\R)$. The possible codimension 1 boundary strata can arise in three different ways.
\begin{enumerate}
\item A connected subset $A=(x_i < x_{i+1} < \ldots x_{i+k-1})$ of points collapsing into single point; 
A limit point
\[
p
\longrightarrow  \tilde{p} = (a_1 < a_2 < \ldots a_{i-1} < a_i = a_{i+1} = \ldots = a_{i+k-1} < a_{i+k}< \ldots < a_n). 
\] 
Points of this type can be identified with $ \overline{Conf}_{n-k+1}(\R) \times \overline{C}_k(\R).$
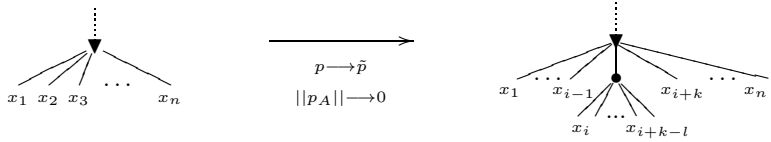
$$
\xymatrix{ 
\xy
(-10,-7)*{_{x_1}},
(-6,-7)*{_{x_2}},
(-2,-7)*{_{x_3}},
(3,-5)*{\ldots},
(10,-7)*{_{x_n}},
(0,0)*{\blacktriangledown}="a",
(0,5)*{}="0",
(-10,-5)*{}="b_1",
(-6,-5)*{}="b_2",
(-2,-5)*{}="b_4",
(10,-5)*{}="b_3",
\ar @{.} "a";"0" <0pt>
\ar @{-} "a";"b_1" <0pt>
\ar @{-} "a";"b_2" <0pt>
\ar @{-} "a";"b_3" <0pt>
\ar @{-} "a";"b_4" <0pt>
\endxy & \ar@{->}[rr]_{\begin{array}{c} _{p \longrightarrow \tilde{p}} \\ _{||p_A|| \longrightarrow 0}\end{array}} &  &  & 
\begin{xy}
<0mm,0mm>*{\blacktriangledown},
<0mm,0.8mm>*{};<0mm,5mm>*{}**@{.},
<-9mm,-5mm>*{\ldots},
<-9mm,-7mm>*{_{x_1\ \,   \ \   \, x_{i-1}}},
<13mm,-7mm>*{_{x_{i+k}\ \, \ \ \  \, x_n}},
<0mm,-10mm>*{...},
<14mm,-5mm>*{\ldots},
<-0.7mm,-0.3mm>*{};<-13mm,-5mm>*{}**@{-},
<-0.6mm,-0.5mm>*{};<-6mm,-5mm>*{}**@{-},
<0.6mm,-0.3mm>*{};<20mm,-5mm>*{}**@{-},
<0.3mm,-0.5mm>*{};<8mm,-5mm>*{}**@{-},
<0mm,-0.5mm>*{};<0mm,-4.3mm>*{}**@{-},
<0mm,-5mm>*{\bullet};
<-5mm,-10mm>*{}**@{-},
<-2.7mm,-10mm>*{}**@{-},
<2.7mm,-10mm>*{}**@{-},
<5mm,-10mm>*{}**@{-},
<2mm,-12mm>*{_{x_{i}}\ \  \ \ _{x_{i+k-l}}},
\end{xy}
}
$$
\item All $n$ points moving in a cluster towards $\pm \infty;$ A limit point $p \longrightarrow \pm (\infty , \infty , \ldots, \infty)$ where the distance between points remain finite, e.g. it could look like $p = (t+\lambda_1,t+\lambda_2, \ldots, t+\lambda_n)$ with $\lambda_1 < \lambda_2 < \ldots <\lambda_n$ and $t \longrightarrow \pm \infty$.  Limit points of this type can be identified with $\widehat{\mathfrak C}_n (\R).$
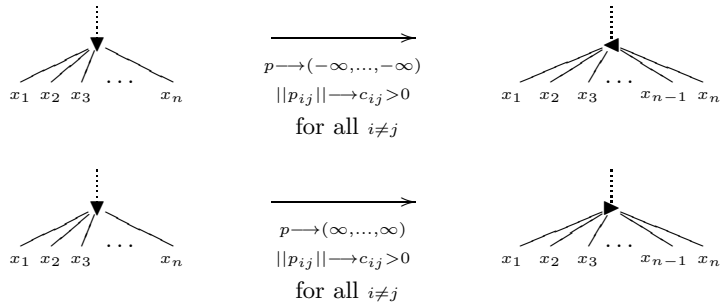
$$
\xymatrix{ 
\xy
(-10,-7)*{_{x_1}},
(-6,-7)*{_{x_2}},
(-2,-7)*{_{x_3}},
(3,-5)*{\ldots},
(10,-7)*{_{x_n}},
(0,0)*{\blacktriangledown}="a",
(0,5)*{}="0",
(-10,-5)*{}="b_1",
(-6,-5)*{}="b_2",
(-2,-5)*{}="b_4",
(10,-5)*{}="b_3",
\ar @{.} "a";"0" <0pt>
\ar @{-} "a";"b_1" <0pt>
\ar @{-} "a";"b_2" <0pt>
\ar @{-} "a";"b_3" <0pt>
\ar @{-} "a";"b_4" <0pt>
\endxy & \ar@{->}[rr]_{\begin{array}{c} _{p \longrightarrow (-\infty, \ldots, -\infty)} \\ _{||p_{ij}|| \longrightarrow c_{ij} > 0} \\ _{\mbox{ for all } i\neq j}\end{array}} &  &  & 
\xy
(1,-5)*{\ldots},
(-13,-7)*{_{x_1}},
(-8,-7)*{_{x_2}},
(-3,-7)*{_{x_3}},
(7,-7)*{_{x_{n-1}}},
(13,-7)*{_{x_n}},
 (0,0)*{\blacktriangleleft}="a",
(0,5)*{}="0",
(-12,-5)*{}="b_1",
(-8,-5)*{}="b_2",
(-3,-5)*{}="b_3",
(8,-5)*{}="b_4",
(12,-5)*{}="b_5",
\ar @{.} "a";"0" <0pt>
\ar @{-} "a";"b_2" <0pt>
\ar @{-} "a";"b_3" <0pt>
\ar @{-} "a";"b_1" <0pt>
\ar @{-} "a";"b_4" <0pt>
\ar @{-} "a";"b_5" <0pt>
\endxy
}
$$
$$
\xymatrix{ 
\xy
(-10,-7)*{_{x_1}},
(-6,-7)*{_{x_2}},
(-2,-7)*{_{x_3}},
(3,-5)*{\ldots},
(10,-7)*{_{x_n}},
(0,0)*{\blacktriangledown}="a",
(0,5)*{}="0",
(-10,-5)*{}="b_1",
(-6,-5)*{}="b_2",
(-2,-5)*{}="b_4",
(10,-5)*{}="b_3",
\ar @{.} "a";"0" <0pt>
\ar @{-} "a";"b_1" <0pt>
\ar @{-} "a";"b_2" <0pt>
\ar @{-} "a";"b_3" <0pt>
\ar @{-} "a";"b_4" <0pt>
\endxy & \ar@{->}[rr]_{\begin{array}{c} _{p \longrightarrow (\infty, \ldots, \infty)} \\ _{||p_{ij}|| \longrightarrow c_{ij} > 0} \\ _{\mbox{ for all } i\neq j}\end{array}} &  &  & 
\xy
(1,-5)*{\ldots},
(-13,-7)*{_{x_1}},
(-8,-7)*{_{x_2}},
(-3,-7)*{_{x_3}},
(7,-7)*{_{x_{n-1}}},
(13,-7)*{_{x_n}},
 (0,0)*{\blacktriangleright}="a",
(0,5)*{}="0",
(-12,-5)*{}="b_1",
(-8,-5)*{}="b_2",
(-3,-5)*{}="b_3",
(8,-5)*{}="b_4",
(12,-5)*{}="b_5",
\ar @{.} "a";"0" <0pt>
\ar @{-} "a";"b_2" <0pt>
\ar @{-} "a";"b_3" <0pt>
\ar @{-} "a";"b_1" <0pt>
\ar @{-} "a";"b_4" <0pt>
\ar @{-} "a";"b_5" <0pt>
\endxy
}
$$
\item For each $k\geq 2$ the $n$ points converge to $k=k_-+1+k_+$ clusters where $k_-$ clusters move to $-\infty,$ $k_+$ clusters move to $+\infty$ and one cluster where each point converge to a finite point. Within each of the $k_-+k_+$ clusters moving to $\pm \infty$ the distance between points remain finite, while the distance from any two points from different clusters tend to $\infty.$ Every such configuration is determined by a disjoint union of connected subsets $A_1 \cup \ldots \cup A_{k_-} \cup F \cup B_1 \cup \ldots \cup B_{k_+}=[n]$ is with $\inf A_1 < \inf A_2 < \ldots < \inf A_{k_-}< \inf F < \inf B_1 < \inf B_2 < \ldots < \inf B_{k_+},$ and limit points of this type can then be identified with $$\overline{C}_k(\R) \times \widehat{\mathfrak C}_{|A_1|}(\R) \times \ldots \times \widehat{\mathfrak C}_{|A_{k_-}|}(\R) \times \overline{Conf}_{|F|} (\R) \times \widehat{\mathfrak C}_{|B_1|}(\R) \times \ldots \times \widehat{\mathfrak C}_{|B_{k_+}|}(\R).
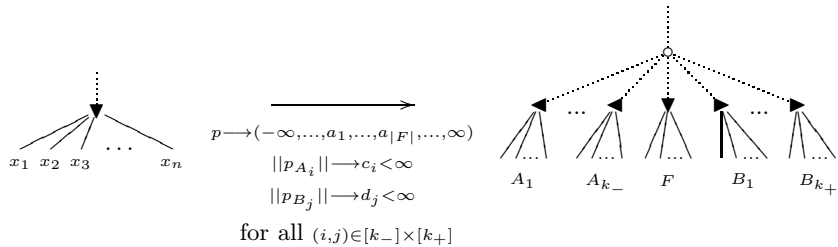
$$
$$
\xymatrix{ 
\xy
(-10,-7)*{_{x_1}},
(-6,-7)*{_{x_2}},
(-2,-7)*{_{x_3}},
(3,-5)*{\ldots},
(10,-7)*{_{x_n}},
(0,0)*{\blacktriangledown}="a",
(0,5)*{}="0",
(-10,-5)*{}="b_1",
(-6,-5)*{}="b_2",
(-2,-5)*{}="b_4",
(10,-5)*{}="b_3",
\ar @{.} "a";"0" <0pt>
\ar @{-} "a";"b_1" <0pt>
\ar @{-} "a";"b_2" <0pt>
\ar @{-} "a";"b_3" <0pt>
\ar @{-} "a";"b_4" <0pt>
\endxy & \ar@{->}[rr]_{\begin{array}{c} _{p \longrightarrow (-\infty, \ldots, a_1,\ldots,a_{|F|},\ldots,\infty)} \\ _{||p_{A_i}|| \longrightarrow c_{i} < \infty} \\_{||p_{B_j}|| \longrightarrow d_{j} < \infty} \\ _{\mbox{ for all } (i,j) \in [k_-] \times [k_+] }\end{array}} &  &  & 
\begin{xy}
(18,-1)*{...},
(42,-1)*{...},
(12,-7)*{_{\ldots}},
(22,-7)*{_{\ldots}},
(31,-7)*{_{\ldots}},
(41,-7)*{_{\ldots}},
(50,-7)*{_{\ldots}},
(11,-10)*{_{A_1}},
(22,-10.5)*{_{A_{k_-}}},
(30,-10)*{_{F}},
(40,-10)*{_{B_1}},
(50,-10.5)*{_{B_{k_+}}},
(30,7)*{\circ}="a",
(13,0)*{\blacktriangleleft}="b_0",
(23,0)*{\blacktriangleleft}="b_3",
(30,0)*{\blacktriangledown}="b_4",
(37,0)*{\blacktriangleright}="b_5",
(47,0)*{\blacktriangleright}="b_7",
(30,13)*{}="0",
(8,-7)*{}="d_1",
(10,-7)*{}="d_3",
(14,-7)*{}="d_2",
(18,-7)*{}="e_1",
(20,-7)*{}="e_3",
(24,-7)*{}="e_2",
(27,-7)*{}="f_1",
(29,-7)*{}="f_3",
(33,-7)*{}="f_2",
(37,-7)*{}="g_1",
(39,-7)*{}="g_2",
(43,-7)*{}="g_3",
(46,-7)*{}="i_1",
(48,-7)*{}="i_2",
(52,-7)*{}="i_3",
\ar @{.} "a";"0" <0pt>
\ar @{.} "a";"b_0" <0pt>
\ar @{.} "a";"b_3" <0pt>
\ar @{.} "a";"b_4" <0pt>
\ar @{.} "a";"b_5" <0pt>
\ar @{.} "a";"b_7" <0pt>
\ar @{-} "b_0";"d_1" <0pt>
\ar @{-} "b_0";"d_2" <0pt>
\ar @{-} "b_0";"d_3" <0pt>
\ar @{-} "b_3";"e_1" <0pt>
\ar @{-} "b_3";"e_2" <0pt>
\ar @{-} "b_3";"e_3" <0pt>
\ar @{-} "b_4";"f_1" <0pt>
\ar @{-} "b_4";"f_2" <0pt>
\ar @{-} "b_4";"f_3" <0pt>
\ar @{-} "b_5";"g_1" <0pt>
\ar @{-} "b_5";"g_2" <0pt>
\ar @{-} "b_5";"g_3" <0pt>
\ar @{-} "b_7";"i_1" <0pt>
\ar @{-} "b_7";"i_3" <0pt>
\ar @{-} "b_7";"i_2" <0pt>
\end{xy}
}
$$
\end{enumerate} 

\end{proof}
\begin{subsection}{The space $\overline{Conf}_n(\R)$ as a smooth manifold with corners.}
We shall endow the space $\overline{Conf}_n(\R)$ with a manifold structure in an almost identical procedure to how the space $\widehat{\mathfrak C}_n(\R)$ was treated. For every tree $t \in \mathcal Ho (As)_\infty$ we define the sets $vert_{\bullet,\circ}(t)$, $vert_{\blacktriangleleft,\blacktriangleright}(t)$ and $vert_{\blacktriangledown}$ as the vertices of $t$ marked by $\{\bullet,\circ\},$ $\{\blacktriangleleft,\blacktriangleright\}$ or $\blacktriangledown$, respectively. For the tree $t$ we define $Conf_t(\R)$ as a product; $$Conf_t(\R) := \prod_{v \in vert_{\bullet,\circ} (t)} C_{|in(v)|}(\R)  \times \prod_{v \in vert_{\blacktriangleleft,\blacktriangleright}(t)} \mathfrak C_{|in(v)|} (\R) \times \prod_{v \in vert_{\blacktriangledown}} Conf_n(\R).$$
We can describe the space $\overline{Conf}_n(\R)$ as a stratified union of spaces; $$\overline{Conf}_n(\R) = \prod_{t \in \mathcal Ho(As)_\infty(n)} Conf_t(\R).$$
We shall define a coordinate chart $U_t$ around every boundary stratum $Conf_t(\R)$ with a metric tree. We associate to $t$ the metric tree $t_{metric}$ with for
\begin{enumerate}
\item  every internal edge of the types $\xy (0,3)*{\blacktriangleleft}="u", (0,-3)*{\bullet}="n", \ar @{-} "u";"n" <0pt> \endxy$, $\xy (0,3)*{\blacktriangleright}="u", (0,-3)*{\bullet}="n", \ar @{-} "u";"n" <0pt> \endxy $ or $\xy (0,3)*{\blacktriangledown}="u", (0,-3)*{\bullet}="n", \ar @{-} "u";"n" <0pt> \endxy $ a small positive parameter $\epsilon;$
\item  every vertex of a dashed corolla associate a large positive number $\tau,$ 
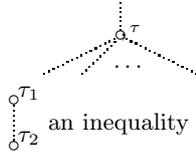
\[\xy (0,0)*{\circ}="o", (2,1)*{_\tau}, (0,5)*{}="u", (-10,-5)*{}="l_1", (-5,-5)*{}="l_2", (1,-4)*{\ldots}="l_3", (10,-5)*{}="l_4" \ar @{.}, "o";"u" <0pt> \ar @{.} "o";"l_1" <0pt> \ar @{.} "o";"l_2" <0pt> \ar @{.} "o";"l_4" <0pt>,  \endxy\]
\item  every subgraph of $t_{metric}$ of the type $ \xy (0,3)*{\circ}="u", (0,-3)*{\circ}="n",(2,4)*{\tau_1}, (2,-2)*{\tau_2}, \ar @{.} "u";"n"<0pt> \endxy$ an inequality $\tau_1 > \tau_2.$
\end{enumerate} 
\begin{example} We consider a specific tree and associate the metric tree to it. The general method should be clear from this description. Let $t$ be the following tree 
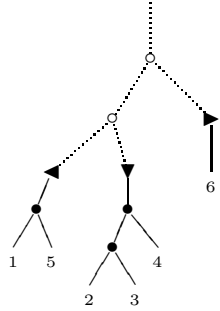
\[
\xy
(0,23)*{}="0",
(0,15)*{\circ}="a",
(-5,7)*{\circ}="b_1",
(8,7)*{\blacktriangleright}="b_2",
(-13,0)*{\blacktriangleleft}="c_1",
(-3,0)*{\blacktriangledown}="c_2",
(-15,-5)*{\bullet}="d_1",
(-3,-5)*{\bullet}="d_2",
(-5,-10)*{\bullet}="e",
(-18,-10)*{}="1",
(-13,-10)*{}="5",
(-8,-15)*{}="2",
(-2,-15)*{}="3",
(1,-10)*{}="4",
(8,0)*{}="6",
(-18,-12)*{_1},
(-13,-12)*{_5},
(-8,-17)*{_2},
(-2,-17)*{_3},
(1,-12)*{_4},
(8,-2)*{_6},
\ar @{.} "0";"a" <0pt>
\ar @{.} "b_1";"a" <0pt>
\ar @{.} "b_2";"a" <0pt>
\ar @{-} "b_2";"6" <0pt>
\ar @{.} "b_1";"c_1" <0pt>
\ar @{.} "b_1";"c_2" <0pt>
\ar @{-} "c_1";"d_1" <0pt>
\ar @{-} "c_2";"d_2" <0pt>
\ar @{-} "d_1";"1" <0pt>
\ar @{-} "d_1";"5" <0pt>
\ar @{-} "d_2";"e" <0pt>
\ar @{-} "d_2";"4" <0pt>
\ar @{-} "e";"2" <0pt>
\ar @{-} "e";"3" <0pt>
\endxy
\]
Then the associated metric tree, $t_{metric},$ is given by 
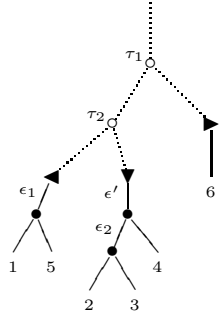
\[
\xy
(0,23)*{}="0",
(0,15)*{\circ}="a",
(-5,7)*{\circ}="b_1",
(-2,16)*{_{\tau_1}}="",
(-7,8)*{_{\tau_2}}="",
(8,7)*{\blacktriangleright}="b_2",
(-13,0)*{\blacktriangleleft}="c_1",
(-16,-2)*{_{\epsilon_1}}="",
(-3,0)*{\blacktriangledown}="c_2",
(-5,-2)*{_{\epsilon'}}="",
(-15,-5)*{\bullet}="d_1",
(-3,-5)*{\bullet}="d_2",
(-6,-7)*{_{\epsilon_2}}="",
(-5,-10)*{\bullet}="e",
(-18,-10)*{}="1",
(-13,-10)*{}="5",
(-8,-15)*{}="2",
(-2,-15)*{}="3",
(1,-10)*{}="4",
(8,0)*{}="6",
(-18,-12)*{_1},
(-13,-12)*{_5},
(-8,-17)*{_2},
(-2,-17)*{_3},
(1,-12)*{_4},
(8,-2)*{_6},
\ar @{.} "0";"a" <0pt>
\ar @{.} "b_1";"a" <0pt>
\ar @{.} "b_2";"a" <0pt>
\ar @{-} "b_2";"6" <0pt>
\ar @{.} "b_1";"c_1" <0pt>
\ar @{.} "b_1";"c_2" <0pt>
\ar @{-} "c_1";"d_1" <0pt>
\ar @{-} "c_2";"d_2" <0pt>
\ar @{-} "d_1";"1" <0pt>
\ar @{-} "d_1";"5" <0pt>
\ar @{-} "d_2";"e" <0pt>
\ar @{-} "d_2";"4" <0pt>
\ar @{-} "e";"2" <0pt>
\ar @{-} "e";"3" <0pt>
\endxy
\]\end{example}

The coordinate chart $U_t \subset \overline{Conf}_n(\R)$ is now defined to be isomorphic to the manifold with corners,
\[
\begin{split}
(l,+\infty]^{|vert_{\circ}(t)|} & \times [0,s)^{|edge_{\bullet}^{\blacktriangleleft,\blacktriangleright,\blacktriangledown}(t)|} \times \prod_{v \in vert_{\circ,\bullet}(t)} C_{|in(v)|}^{st} (\R) \\ & \times \prod_{v \in vert_{\blacktriangleleft,\blacktriangleright}(t)} \mathfrak C_{|in(t)|}^{st}(\R) \times \prod_{v \in vert_{\blacktriangledown}(t)} Conf_{|in(v)|}(\R)
\end{split}
\]
where $vert_{\circ}$ denotes the set of vertices of type $\circ$, $vert_{\circ,\bullet}$ denotes the set of vertices of type $\circ$ or $\bullet$ and so forth. The set $edge_{\bullet}^{\blacktriangleleft,\blacktriangleright, \blacktriangledown}$ is give set of edges of type  $ \xy (0,3)*{\blacktriangleleft}="u", (0,-3)*{\bullet}="n", \ar @{-} "u";"n" <0pt> \endxy,$ $ \xy (0,3)*{\blacktriangleright}="u", (0,-3)*{\bullet}="n", \ar @{-} "u";"n" <0pt> \endxy$ or $ \xy (0,3)*{\blacktriangledown}="u", (0,-3)*{\bullet}="n", \ar @{-} "u";"n" <0pt> \endxy.$ The isomorphism $\Phi_t$ between the coordinate chart $U_t$ and the product above is read from the metric tree. The map is given in coordinates, for the specific tree in the above example, as follows
\[
\begin{array}{cccccccccccccccccccc}
(l,+\infty]^2 & \times & [0,s)^3 & \times & C_{2}^{st} (\R) & \times &C_2^{st}(\R) & \times &C_2^{st}(\R)  \\ (\tau_1,\tau_2) & \times & (\epsilon_1,\epsilon_2,\epsilon') & \times & (x''_1,x''_2) & \times & (x'_1,x'_2) & \times & (x_1,x_5) \\&\\ C_2^{st}(\R) & \times & C_2^{st}(\R) & \times &  \mathfrak C_{1}^{st}(\R) & \times &\mathfrak C_{1}^{st}(\R) & \times & Conf_1(\R)& \\   (x',x_4) & \times & (x_2,x_3) & \times & x_6 & \times & s & \times & u\\ &&&&&&&&& \longrightarrow & Conf_6(\R) \\ &&&&& &&&&& (y_1, \ldots, y_6)
\end{array}
\]

such that 
\[
\begin{array}{cccc}
y_1=&\tau_1 x''_1+\tau_2 x_1 +t +\epsilon_1 x_1 & y_2 =& \tau_1 x''+\tau_2 x' + u+\epsilon' (x'+\epsilon_2 x_2)  \\
y_3= &\tau_1 x''+\tau_2 x' + u+\epsilon' (x'+\epsilon_2 x_3)  & y_4 =& \tau_1 x''+\tau_2 x' + u+\epsilon' x_4 \\
y_5=&\tau_1 x''_1+\tau_2 x_1 +t +\epsilon_1 x_5 & y_6 = &\tau_1 x''+ x_6 
\end{array}
\]

The boundary strata in $U_t$ are given by allowing formally $\tau_1=\infty, \tau_2 = \infty$ such that $\tau_1/\tau_2 =0$ and $\epsilon_1 =0,\epsilon_2=0,$ $\epsilon'=0.$
\end{subsection} 
\begin{subsection}{The Cohomology of $\mathcal Ho(As)_\infty$}
We will need two results in order to calculate the cohomology of the operad $\mathcal Ho(As)_\infty.$ 
\begin{theorem}
Let $P$ be a koszul operad. Define the two colored operad $\mathcal Mor(P)$ whose representations are two $P$-algebras and a $P$-algebra morphism between them. The operad $\mathcal Mor(P)$ has a minimal model given by the operad $\mathcal Mor(P)_\infty$ whose representations are two $P_\infty$-algebras and a homotopy $P_\infty$-morphism between them.
\end{theorem}
\begin{proof}
See \cite{MV09a},\cite{MV09b}.
\end{proof}
\begin{corollary}\label{Mor(As)}
The operad $\mathcal Mor(As)$ has a minimal model given by $\mathcal Mor(As)_\infty.$
\end{corollary} 

\begin{lemma} \label{comparison}
Let $f:B \to C$ be a map of filtered complexes, where both $B$ and $C$ are complete and exhaustive. Fix $r\geq 0.$ Suppose that $f^r:E_{pq}^r(B) \cong E_{pq}^r(C)$ for all $p$ and $q$. Then $f:\mathrm H(B) \to \mathrm H(C)$ is an isomorphism.
\end{lemma}
This result is known as the comparison lemma, and can be found in a textbook on homological algebra, e.g. \cite{We}.

We can now state our result.

\begin{theorem}
The cohomology of $\mathcal Ho(As)_\infty$ is the operad $\mathcal Mor(As)$ whose representations are a pair of associative algebras and a morphism of associative algebras between them.
\end{theorem}
\begin{proof} There is a natural projection of operads $$\pi: \mathcal Ho(As)_\infty \twoheadrightarrow \mathcal Mor(As)_\infty.$$ We can describe this map on corollas by using the presentation of $\mathcal Ho(As)_\infty$ and $\mathcal Mor(As)_\infty$ from theorem \ref{hoass} and section \ref{confmorass}, respectively; 
\[
\begin{split} \pi \left( \xy
(1,-5)*{\ldots},
 (0,0)*{\blacktriangledown}="a",
(0,5)*{}="0",
(-12,-5)*{}="b_1",
(-8,-5)*{}="b_2",
(-3,-5)*{}="b_3",
(8,-5)*{}="b_4",
(12,-5)*{}="b_5",
\ar @{-} "a";"0" <0pt>
\ar @{-} "a";"b_2" <0pt>
\ar @{-} "a";"b_3" <0pt>
\ar @{-} "a";"b_1" <0pt>
\ar @{-} "a";"b_4" <0pt>
\ar @{-} "a";"b_5" <0pt>
\endxy \right) &= 0 
\\ 
\pi\left( \xy
(1,-5)*{\ldots},
 (0,0)*{\blacktriangleleft}="a",
(0,5)*{}="0",
(-12,-5)*{}="b_1",
(-8,-5)*{}="b_2",
(-3,-5)*{}="b_3",
(8,-5)*{}="b_4",
(12,-5)*{}="b_5",
\ar @{-} "a";"0" <0pt>
\ar @{-} "a";"b_2" <0pt>
\ar @{-} "a";"b_3" <0pt>
\ar @{-} "a";"b_1" <0pt>
\ar @{-} "a";"b_4" <0pt>
\ar @{-} "a";"b_5" <0pt>
\endxy \right) 
& = \pi \left(  \xy
(1,-5)*{\ldots}, 
(0,0)*{\blacktriangleright}="a",
(0,5)*{}="0",
(-12,-5)*{}="b_1",
(-8,-5)*{}="b_2",
(-3,-5)*{}="b_3",
(8,-5)*{}="b_4",
(12,-5)*{}="b_5",
\ar @{-} "a";"0" <0pt>
\ar @{-} "a";"b_2" <0pt>
\ar @{-} "a";"b_3" <0pt>
\ar @{-} "a";"b_1" <0pt>
\ar @{-} "a";"b_4" <0pt>
\ar @{-} "a";"b_5" <0pt>
\endxy \right) = \xy
(1,-5)*{\ldots}, 
(0,0)*{\blacksquare}="a",
(0,5)*{}="0",
(-12,-5)*{}="b_1",
(-8,-5)*{}="b_2",
(-3,-5)*{}="b_3",
(8,-5)*{}="b_4",
(12,-5)*{}="b_5",
\ar @{-} "a";"0" <0pt>
\ar @{-} "a";"b_2" <0pt>
\ar @{-} "a";"b_3" <0pt>
\ar @{-} "a";"b_1" <0pt>
\ar @{-} "a";"b_4" <0pt>
\ar @{-} "a";"b_5" <0pt>
\endxy 
\\ 
\pi \left(
\xy
(1,-5)*{\ldots}, 
(0,0)*{\bullet}="a",
(0,5)*{}="0",
(-12,-5)*{}="b_1",
(-8,-5)*{}="b_2",
(-3,-5)*{}="b_3",
(8,-5)*{}="b_4",
(12,-5)*{}="b_5",
\ar @{-} "a";"0" <0pt>
\ar @{-} "a";"b_2" <0pt>
\ar @{-} "a";"b_3" <0pt>
\ar @{-} "a";"b_1" <0pt>
\ar @{-} "a";"b_4" <0pt>
\ar @{-} "a";"b_5" <0pt>
\endxy \right) 
&= 
\xy
(1,-5)*{\ldots},
 (0,0)*{\bullet}="a",
(0,5)*{}="0",
(-12,-5)*{}="b_1",
(-8,-5)*{}="b_2",
(-3,-5)*{}="b_3",
(8,-5)*{}="b_4",
(12,-5)*{}="b_5",
\ar @{-} "a";"0" <0pt>
\ar @{-} "a";"b_2" <0pt>
\ar @{-} "a";"b_3" <0pt>
\ar @{-} "a";"b_1" <0pt>
\ar @{-} "a";"b_4" <0pt>
\ar @{-} "a";"b_5" <0pt>
\endxy 
\\
\pi \left( 
\xy
(1,-5)*{\ldots}, 
(0,0)*{\circ}="a",
(0,5)*{}="0",
(-12,-5)*{}="b_1",
(-8,-5)*{}="b_2",
(-3,-5)*{}="b_3",
(8,-5)*{}="b_4",
(12,-5)*{}="b_5",
\ar @{-} "a";"0" <0pt>
\ar @{-} "a";"b_2" <0pt>
\ar @{-} "a";"b_3" <0pt>
\ar @{-} "a";"b_1" <0pt>
\ar @{-} "a";"b_4" <0pt>
\ar @{-} "a";"b_5" <0pt>
\endxy
\right)
&=
\xy
(1,-5)*{\ldots}, 
(0,0)*{\circ}="a",
(0,5)*{}="0",
(-12,-5)*{}="b_1",
(-8,-5)*{}="b_2",
(-3,-5)*{}="b_3",
(8,-5)*{}="b_4",
(12,-5)*{}="b_5",
\ar @{-} "a";"0" <0pt>
\ar @{-} "a";"b_2" <0pt>
\ar @{-} "a";"b_3" <0pt>
\ar @{-} "a";"b_1" <0pt>
\ar @{-} "a";"b_4" <0pt>
\ar @{-} "a";"b_5" <0pt>
\endxy
\end{split} 
\]  
The map $\pi$ obviously respect the differentials of the operads.

We introduce a filtration on $\mathcal Ho(As)_\infty(n)$ and $\mathcal Mor(As)_\infty(n)$ on the number of internal vertices in a tree, $$F_P \mathcal Ho(As)_\infty(n) = \{x\in \mathcal Ho(As)_\infty(n) | \mbox{number of internal vertices of }x \geq p \}$$
and
$$F_P \mathcal Mor(As)_\infty(n) = \{x\in \mathcal Mor(As)_\infty(n) | \mbox{number of internal vertices of }x \geq p \}.$$
Clearly the differentials in $\mathcal Ho(As)_\infty$ and  $\mathcal Mor(As)_\infty$ respect these filtrations as the number of vertices can only stay the same or increase when the differentials are applied. Note that the filtrations are both exhaustive and complete, this follows from that the objects in question are finite dimensional for any given $n.$ The induced differential on $E_{pq}^0( \mathcal Mor(As)_\infty)$ will either map a corolla to zero or increase the number of vertices and therefore $E_{pq}^1( \mathcal Mor(As)_\infty) = \mathrm H(E_{pq}^0( \mathcal Mor(As)_\infty)) = E_{pq}^0( \mathcal Mor(As)_\infty).$ On the other hand, in the case of $E_{pq}^0( \mathcal Ho(As)_\infty),$ we have that the differential will map all trees except those containing a corolla of type 

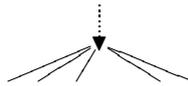
\[\xy
%
 (0,0)*{\blacktriangledown}="a",
(0,5)*{}="0",
(-12,-5)*{}="b_1",
(-8,-5)*{}="b_2",
(-3,-5)*{}="b_3",
(8,-5)*{}="b_4",
(12,-5)*{}="b_5",
\ar @{.} "a";"0" <0pt>
\ar @{-} "a";"b_2" <0pt>
\ar @{-} "a";"b_3" <0pt>
\ar @{-} "a";"b_1" <0pt>
\ar @{-} "a";"b_4" <0pt>
\ar @{-} "a";"b_5" <0pt>
\endxy\]

to zero. We get that the image of $\partial^0: E^0_{pq}( \mathcal Ho( As) _ \infty) \to E^0_{pq}( \mathcal Ho(As) _ \infty)$ will consist of trees (operadically) generated by the difference of corollas; 

\[\left\langle \xy
%
 (0,0)*{\blacktriangleleft}="a",
(0,5)*{}="0",
(-12,-5)*{}="b_1",
(-8,-5)*{}="b_2",
(-3,-5)*{}="b_3",
(8,-5)*{}="b_4",
(12,-5)*{}="b_5",
\ar @{.} "a";"0" <0pt>
\ar @{-} "a";"b_2" <0pt>
\ar @{-} "a";"b_3" <0pt>
\ar @{-} "a";"b_1" <0pt>
\ar @{-} "a";"b_4" <0pt>
\ar @{-} "a";"b_5" <0pt>
\endxy - \xy
%
 (0,0)*{\blacktriangleright}="a",
(0,5)*{}="0",
(-12,-5)*{}="b_1",
(-8,-5)*{}="b_2",
(-3,-5)*{}="b_3",
(8,-5)*{}="b_4",
(12,-5)*{}="b_5",
\ar @{.} "a";"0" <0pt>
\ar @{-} "a";"b_2" <0pt>
\ar @{-} "a";"b_3" <0pt>
\ar @{-} "a";"b_1" <0pt>
\ar @{-} "a";"b_4" <0pt>
\ar @{-} "a";"b_5" <0pt> 
\endxy
\right\rangle \subset \mathcal Ho(\mathcal A _\infty).
\]

The first page is then determined; 
\[
\begin{split}
 & E_{pq}^1 ( \mathcal Ho(As)_\infty) \\ 
& = 
\frac{\mathcal Free \left\langle \xy
(1,-5)*{\ldots},
 (0,0)*{\bullet}="a",
(0,5)*{}="0",
(-12,-5)*{}="b_1",
(-8,-5)*{}="b_2",
(-3,-5)*{}="b_3",
(8,-5)*{}="b_4",
(12,-5)*{}="b_5",
\ar @{-} "a";"0" <0pt>
\ar @{-} "a";"b_2" <0pt>
\ar @{-} "a";"b_3" <0pt>
\ar @{-} "a";"b_1" <0pt>
\ar @{-} "a";"b_4" <0pt>
\ar @{-} "a";"b_5" <0pt>
\endxy, \xy
(1,-5)*{\ldots},
 (0,0)*{\circ}="a",
(0,5)*{}="0",
(-12,-5)*{}="b_1",
(-8,-5)*{}="b_2",
(-3,-5)*{}="b_3",
(8,-5)*{}="b_4",
(12,-5)*{}="b_5",
\ar @{.} "a";"0" <0pt>
\ar @{.} "a";"b_2" <0pt>
\ar @{.} "a";"b_3" <0pt>
\ar @{.} "a";"b_1" <0pt>
\ar @{.} "a";"b_4" <0pt>
\ar @{.} "a";"b_5" <0pt>
\endxy, \xy
(1,-5)*{\ldots},
 (0,0)*{\blacktriangleleft}="a",
(0,5)*{}="0",
(-12,-5)*{}="b_1",
(-8,-5)*{}="b_2",
(-3,-5)*{}="b_3",  
(8,-5)*{}="b_4",
(12,-5)*{}="b_5",
\ar @{.} "a";"0" <0pt>
\ar @{-} "a";"b_2" <0pt>
\ar @{-} "a";"b_3" <0pt>
\ar @{-} "a";"b_1" <0pt>
\ar @{-} "a";"b_4" <0pt>
\ar @{-} "a";"b_5" <0pt>
\endxy, \xy
(1,-5)*{\ldots},
 (0,0)*{\blacktriangleright}="a",
(0,5)*{}="0",
(-12,-5)*{}="b_1",
(-8,-5)*{}="b_2",
(-3,-5)*{}="b_3",
(8,-5)*{}="b_4",
(12,-5)*{}="b_5",
\ar @{.} "a";"0" <0pt>
\ar @{-} "a";"b_2" <0pt>
\ar @{-} "a";"b_3" <0pt>
\ar @{-} "a";"b_1" <0pt>
\ar @{-} "a";"b_4" <0pt>
\ar @{-} "a";"b_5" <0pt>
\endxy \right\rangle } 
{ \left\langle \xy
(1,-5)*{\ldots},
%
 (0,0)*{\blacktriangleleft}="a",
(0,5)*{}="0",
(-12,-5)*{}="b_1",
(-8,-5)*{}="b_2",
(-3,-5)*{}="b_3",
(8,-5)*{}="b_4",
(12,-5)*{}="b_5",
\ar @{.} "a";"0" <0pt>
\ar @{-} "a";"b_2" <0pt>
\ar @{-} "a";"b_3" <0pt>
\ar @{-} "a";"b_1" <0pt>
\ar @{-} "a";"b_4" <0pt>
\ar @{-} "a";"b_5" <0pt>
\endxy - \xy
(1,-5)*{\ldots},
%
 (0,0)*{\blacktriangleright}="a",
(0,5)*{}="0",
(-12,-5)*{}="b_1",
(-8,-5)*{}="b_2",
(-3,-5)*{}="b_3",
(8,-5)*{}="b_4",
(12,-5)*{}="b_5",
\ar @{.} "a";"0" <0pt>
\ar @{-} "a";"b_2" <0pt>
\ar @{-} "a";"b_3" <0pt>
\ar @{-} "a";"b_1" <0pt>
\ar @{-} "a";"b_4" <0pt>
\ar @{-} "a";"b_5" <0pt> 
\endxy
\right\rangle } \\ 
& \cong E_{pq}^0(\mathcal Mor(As)_\infty) \\ 
&=E_{pq}^1 ( \mathcal Mor(As)_\infty)
\end{split}
\]
By the comparison lemma \ref{comparison} we find that the cohomology of $\mathcal Ho(As)_\infty$ is the same as that of $\mathcal Mor(As)_\infty,$ which by the lemma \ref{Mor(As)} is precisely $\mathcal Mor(As).$
\end{proof}
\begin{corollary} \label{model}
The operad $\mathcal Ho(As)_\infty$ is a non-minimal quasi-free model of $\mathcal Ho(As).$
\end{corollary}
\begin{proof} There is a natural projection of operads $$p:\mathcal Ho(As)_\infty \twoheadrightarrow \mathcal Ho(As).$$ We determine the cohomology of the operad $\mathcal Ho(As).$ Let $H(\mathcal Ho(As)) = Z/B,$ then if $\mu_V$ and $\mu_W$ are the multiplications, $f,g: V\to W$ are the algebra morphisms and $h: f \sim g$ is the homotopy between them. We will have that $\partial f =\partial g = \partial \mu_V = \partial \mu_W=0$, so the generators all constitute cycles. The boundaries are generated by $\partial h = f-g.$ Hence $$ \mathrm H (\mathcal Ho(As)) = Z/B = \langle f,g,\mu_V,\mu_W\rangle / (f-g) \cong \langle [f],\mu_V,\mu_W\rangle$$ and we see that the cohomology is equal to $\mathcal Mor(As).$ By the previous theorem the corollary is now implied.
\end{proof}
\begin{corollary}
The operad $\mathcal Ho(As)_\infty$ is a non-minimal quasi-free model of $\mathcal Mor(As).$
\end{corollary}
\begin{proof} The model-structure comes from the map $$\tilde{p}: \mathcal Ho(As)_\infty \twoheadrightarrow \mathcal Mor(As),$$ which is given by post-composing the map $p$ from corollary \ref{model} with the natural projection onto cohomology classes; $$\mathcal Ho(As) \twoheadrightarrow \mathrm H (\mathcal Ho(As)) \cong \mathcal Mor(As).$$
\end{proof}

\end{subsection}
\end{section}

 \end{document}